\documentclass[12pt]{amsart}

\usepackage{amssymb,amsthm,amsmath,amsxtra}
\usepackage{color}
\usepackage[usenames,dvipsnames]{xcolor}
\usepackage[all]{xy}
\usepackage{fullpage}
\usepackage{comment}
\usepackage{colonequals} 
\usepackage{mathrsfs}

\usepackage{enumerate}
\usepackage{hyperref}
\usepackage{graphics}
\usepackage{marginnote}
\usepackage{adjustbox}
\usepackage[section]{placeins}
\usepackage{cancel}
 
\numberwithin{equation}{subsection}

\theoremstyle{plain}
\newtheorem{thm}[equation]{Theorem}
\newtheorem{prop}[equation]{Proposition}
\newtheorem{lem}[equation]{Lemma} 
\newtheorem{cor}[equation]{Corollary}

\newtheorem*{cor*}{Corollary}
\newtheorem*{prob*}{Problem}
\newtheorem*{thm*}{Theorem}
\newtheorem*{thma*}{Theorem A}
\newtheorem*{thmb*}{Theorem B}

\theoremstyle{remark}
\newtheorem{exm}[equation]{Example}
\newtheorem{defn}[equation]{Definition}
\newtheorem{rmk}[equation]{Remark}

\newenvironment{enumalph}
{\begin{enumerate}}
{\end{enumerate}}

\DeclareMathOperator{\Aut}{Aut}

\DeclareMathOperator{\Frob}{Frob}
\DeclareMathOperator{\Gal}{Gal}

\DeclareMathOperator{\lcm}{lcm}
\DeclareMathOperator{\GL}{GL}

\DeclareMathOperator{\ord}{ord}
\DeclareMathOperator{\SL}{SL}
\DeclareMathOperator{\supp}{supp}
\DeclareMathOperator{\Tr}{Tr}

\DeclareMathOperator{\Res}{Res}
\DeclareMathOperator{\vol}{vol}

\newcommand{\dt}{d^T}

\newcommand{\Fsf}{\mathsf F}
\newcommand{\Lsf}{\mathsf L}
\newcommand{\Csf}{\mathsf C}

\newcommand{\defi}[1]{\emph{\textsf{#1}}} 				

\newcommand{\A}{\mathbb A}
\newcommand{\C}{\mathbb C}
\newcommand{\F}{\mathbb F}
\newcommand{\G}{\mathbb G}

\newcommand{\PP}{\mathbb P}
\newcommand{\Q}{\mathbb Q}
\newcommand{\R}{\mathbb R}
\newcommand{\Z}{\mathbb Z}

\newcommand{\calF}{\mathcal{F}}

\newcommand{\dD}{\mathrm{d}}

\newcommand{\psmod}[1]{~(\textup{\text{mod}}~{#1})}

\begin{document}

\title[Zeta functions]{Zeta functions of alternate mirror Calabi--Yau families}

\author{Charles F. Doran}
\address{University of Alberta, Department of Mathematics, Edmonton, AB Canada}
\email{doran@math.ualberta.edu}

\author{Tyler L. Kelly}
\address{Department of Pure Mathematics and Mathematical Statistics, University of Cambridge, Wilberforce Road, Cambridge CB3 0WB, UK}
\email{tlk20@dpmms.cam.ac.uk}

\author{Adriana Salerno}
\address{Department of Mathematics, Bates College, 3 Andrews Rd., Lewiston, ME 04240, USA}
\email{asalerno@bates.edu}

\author{Steven Sperber}
\address{School of Mathematics, University of Minnesota, 206 Church Street SE, Minneapolis, MN 55455 USA}
\email{sperber@umn.edu}

\author{John Voight}
\address{Department of Mathematics, Dartmouth College, 6188 Kemeny Hall, Hanover, NH 03755, USA}
\email{jvoight@gmail.com}

\author{Ursula Whitcher}
\address{Mathematical Reviews, 416 Fourth St, Ann Arbor, MI 48103, USA}
\email{uaw@umich.edu}

\date{\today}
\setcounter{tocdepth}{1}

\begin{abstract}
We prove that if two Calabi--Yau invertible pencils have the same dual weights, then they share a common factor in their zeta functions. By using Dwork cohomology, we demonstrate that this common factor is related to a hypergeometric Picard--Fuchs differential equation. The factor in the zeta function is defined over the rationals and has degree at least the order of the Picard--Fuchs equation. As an application, we relate several pencils of K3 surfaces to the Dwork pencil, obtaining new cases of arithmetic mirror symmetry.
\end{abstract}

\maketitle

\section{Introduction}

\subsection{Motivation}

For a variety $X$ over a finite field $\F_q$, the zeta function of $X$ is the exponential generating function for the number of $\F_{q^r}$-rational points, given by
\[ Z(X,T) \colonequals \exp\left(\sum_{r=1}^{\infty}\frac{\#X(\F_{q^r})T^r}{r}\right) \in \Q(T). \]

\noindent In his study of the Weil conjectures, Dwork analyzed the way the zeta function varies for one-parameter deformations of Fermat hypersurfaces in projective space, like the pencil
\begin{equation} \label{eqn:dwork}
x_0^{n+1}+\dots+ x_n^{n+1} - (n+1)\psi x_0x_1\cdots x_n = 0
\end{equation}
in the parameter $\psi$.  In his 1962 ICM address \cite{Dwo62}, Dwork constructed a family of endomorphisms whose characteristic polynomials determined the zeta functions of the hypersurfaces modulo $p$. Furthermore, he identified a power series in the deformation parameter with rational function coefficients that satisfies an ordinary differential equation with regular singular points.  In fact, this differential equation is the Picard--Fuchs equation for the  holomorphic differential form \cite{Kat68}.  The pencil \eqref{eqn:dwork} is a central example in both arithmetic and algebraic geometry \cite{katz:dwork}; we label this family $\Fsf_{n+1}$.

On the arithmetic side, Dwork \cite{padic} analyzed $\Fsf_4$ in detail to explore the relationship between the Picard--Fuchs differential equation satisfied by the holomorphic form on the family and the characteristic polynomial of Frobenius acting on middle-dimensional cohomology.  Dwork identifies the reciprocal zeros of the zeta function for this family of K3 surfaces explicitly by studying $p$-adic solutions of the Picard--Fuchs equation.  This analysis motivated Dwork's general study of $p$-adic periods.

On the algebraic side, the family of Calabi--Yau threefolds $\Fsf_5$ has been used to explore the deep geometric relationship known as \emph{mirror symmetry}.  Mirror symmetry is a duality from string theory that has shaped research in geometry and physics for the last quarter-century. Loosely defined, it predicts a duality where, given a Calabi--Yau variety $X$ there exists another Calabi--Yau variety $Y$, the mirror, so that various geometric and physical data is exchanged. For example, Candelas--de la Ossa--Green--Parkes \cite{CDGP} showed that the number of rational curves on quintic threefolds in projective space can be computed by studying the mirror family, realized via the Greene--Plesser mirror construction \cite{GP} as a resolution of a finite quotient of $\Fsf_5$.

Combining both sides, Candelas, de la Ossa, and Rodriguez-Villegas used the Greene--Plesser mirror construction and techniques from toric varieties to compare the zeta function of fibers $X_\psi$ of $\Fsf_5$ and the mirror pencil of threefolds $Y_\psi$ \cite{CORV, CORV2, candelas}.  They found that for general $\psi$, the zeta functions of $X_\psi$ and $Y_\psi$ share a common factor  related to the period of the holomorphic form on $X_\psi$.  In turn, they related the other nontrivial factors of $Z(X_\psi, T)$ to the action of discrete scaling symmetries of the Dwork pencil $\Fsf_5$ on homogeneous monomials.  In related work (but in a somewhat different direction), Jeng-Daw Yu \cite{yu} showed that the unique unit root for the middle-dimensional factor of the zeta function for the Dwork family in dimension $n$ can be expressed in terms of a ratio of holomorphic solutions of a hypergeometric Picard--Fuchs equation (evaluated at certain values).

The Dwork pencil $\Fsf_5$ is not the only highly symmetric pencil that may be used to construct the mirror to quintic threefolds.  In fact, there are six different pencils of projective Calabi--Yau threefolds, each admitting a different group action, that yield such a mirror: these pencils were studied by Doran--Greene--Judes \cite{DGJ08} at the level of Picard--Fuchs equations. Bini--van Geemen--Kelly \cite{BvGK} then studied the Picard--Fuchs equations for alternate pencils in all dimensions.    

A general mechanism for finding alternate mirrors is given by the framework of \emph{Berglund--H\"{u}bsch--Krawitz (BHK) duality}.  This framework identifies the mirrors of individual Calabi--Yau varieties given by \emph{invertible polynomials}, or more generally of \emph{invertible pencils}, the one-parameter monomial deformation of invertible polynomials; these notions are made precise in the next section.  Aldi--Peruni\v{c}i\'c \cite{AP15} have studied the arithmetic nature of invertible polynomials via D-modules. 

In this paper, we show that invertible pencils whose mirrors have common properties share arithmetic similarities as well. Revisiting work of G\"ahrs \cite{GahrsThesis}, we find that invertible pencils whose BHK mirrors are hypersurfaces in quotients of the same weighted projective space have the same Picard--Fuchs equation associated to their holomorphic form. In turn, we show that the Picard--Fuchs equations for the pencil dictate a factor of the zeta functions of the pencil.  We then show that the factor of the zeta function is bounded by the degree of the Picard--Fuchs equation and the dimension of the piece of the middle cohomology that is invariant under the action of a finite group of symmetries fixing the holomorphic form.

\subsection{Main theorem} 

An \defi{invertible polynomial} is a polynomial of the form
$$
F_{A} = \sum_{i=0}^n \prod_{j=0}^n x_j^{a_{ij}}  \in \Z[x_0,\dots,x_n],
$$
where the matrix of exponents $A = (a_{ij})_{i,j}$ is an $(n+1) \times (n+1)$ matrix with nonnegative integer entries, such that:
\begin{itemize}
\item $\det(A) \neq 0$,
\item there exist $r_0,\dots,r_n \in \Z_{>0}$ and $d \in \Z$ such that $\sum_{j=0}^n r_j a_{ij}=d$ (i.e., the polynomial $F_{A}$ is quasi-homogeneous), and 
\item the function $F_{A}: \mathbb{C}^{n+1} \rightarrow \mathbb{C}$ has exactly one singular point at the origin.
\end{itemize}
We will be particularly interested in the case where $F_{A}$ is invertible and homogeneous of degree $d=n+1$: then the hypersurface defined by $F_A=0$ defines a Calabi--Yau variety in $\PP^n$.  

These conditions are restrictive. In fact, Kreuzer--Skarke \cite{KS} proved  that any invertible polynomial $F_A(x)$ can be written as a sum of polynomials, each of which belongs to one of three \defi{atomic types}, known as \defi{Fermat}, \defi{loop}, and \defi{chain}:
\begin{equation*}
\begin{aligned}
\text{Fermats } &: \quad x^{a}, \\
\text{loops } &: \quad x_1^{a_1}x_2 + x_2^{a_2}x_3 + \ldots +x_{m-1}^{a_{m-1}}x_m + x_m^{a_m}x_1, \text{ and } \\
\text{chains } &: \quad x_1^{a_1}x_2 + x_2^{a_2}x_3 + \ldots +x_{m-1}^{a_{m-1}}x_m + x_m^{a_m}.
\end{aligned}
\end{equation*}
Invertible polynomials appeared as the first families exemplifying mirror symmetry \cite{GP, BH93}. Their arithmetic study, often in the special case of Delsarte polynomials, is of continuing interest \cite{Shi86, EGZ}. 

Let $F_A$ be an invertible polynomial.  Inspired by Berglund--H\"ubsch--Krawitz (BHK) mirror symmetry \cite{BH93, Kra09}, we look at the polynomial obtained from the transposed matrix $A^T$:
$$
F_{A^T} \colonequals   \sum_{i=0}^n \prod_{j=0}^n x_j^{a_{ji}}.
$$
Then $F_{A^T}$ is again an invertible polynomial, quasihomogeneous with (possibly different) weights $q_0, \ldots, q_n$ for which we may assume $\operatorname{gcd}(q_0, \ldots, q_n) =1$, so that $F_{A^T}=0$ defines a hypersurface $X_{A^T}$ in the weighted-projective space $W\mathbb{P}^n(q_0, \ldots, q_n)$.  We call $q_0, \ldots, q_n$ the \defi{dual weights} of $F_A$. Let $\dt \colonequals \sum_i q_i$ be the sum of the dual weights.

We define a one-parameter deformation of our invertible polynomial by
\begin{equation} \label{eqn:invertpencil}
F_{A,\psi} \colonequals \sum_{i=0}^n \prod_{j=0}^n x_j^{a_{ij}} - \dt\psi x_0 \cdots x_n \in \Z[\psi][x_0,\dots,x_n].
\end{equation}
Then $X_{A, \psi}: F_{A,\psi}=0$ is a family of hypersurfaces in $\mathbb{P}^n$ in the parameter $\psi$, which we call an \defi{invertible pencil}.  

The Picard--Fuchs equation for the family $X_{A, \psi}$ is determined completely by the $(n+1)$-tuple of dual weights $(q_0,\dots,q_n)$ by work of G\"ahrs \cite[Theorem 3.6]{GahrsThesis}.  In particular, there is an explicit formula for the order $D(q_0,\dots,q_n)$ of this Picard--Fuchs equation that depends only on the dual weights: see Theorem~\ref{thm: Gahrs} for details.  We further observe that the Picard--Fuchs equation is a hypergeometric differential equation. 

For a smooth projective hypersurface $X$ in $\mathbb{P}^n$, we have
\begin{equation} \label{eqn:PXT}
Z(X, T) = \frac{P_X(T)^{(-1)^n}}{(1-T)(1-qT)\cdots (1-q^{n-1}T)},
\end{equation}
with $P_X(T) \in \Q[T]$.  Our main result is as follows (for the notion of nondegenerate, see section 2).

\begin{thm} \label{thm:sameweights}
Let $X_{A,\psi}$ and $X_{B,\psi}$ be invertible pencils of Calabi--Yau $(n-1)$-folds in $\PP^n$.  Suppose $A$ and $B$ have the same dual weights $(q_i)_i$.  Then for each $\psi \in \F_q$ such that $\gcd(q,(n+1)\dt)=1$ and the fibers $X_{A,\psi}$ and $X_{B,\psi}$ are nondegenerate and smooth, 
the polynomials $P_{X_{A, \psi}}(T)$ and $P_{X_{B,\psi}}(T)$ have a common factor $R_\psi(T) \in \Q[T]$ with 
\[ \deg R_{\psi}(T) \geq D(q_0,\dots,q_n).\]
\end{thm}

We show that the common factor $R_\psi(T)$ is attached to the holomorphic form on $X_{A,\psi}$ and $X_{B,\psi}$, explaining the link to the Picard--Fuchs differential equation: it is given explicitly in terms of a hypergeometric series \eqref{equ:genhypergeometric}.  For this reason, if we had an appropriate theorem for rigidity of hypergeometric motives, we could further conclude that there exists a factor of degree precisely $D(q_0,\dots,q_n)$ in $\Q[T]$: see Remark \ref{rmk:hypergeom}.  For invertible pencils with dual weights $(1,\dots,1)$, including those comprised of only Fermats and loops, we can nail this down precisely (Corollary \ref{cor:dworksolved}).  

\begin{cor}
With hypotheses as in Theorem \textup{\ref{thm:sameweights}}, suppose that the common dual weights are $(q_0,\ldots,q_n)=(1,\ldots,1)$.  Then the common factor $R_\psi(T) \in \Q[T]$ has $\deg R_{\psi}=n$.
\end{cor}

Our proof of Theorem \ref{thm:sameweights} uses the $p$-adic cohomology theory of Dwork, as developed by Adolphson--Sperber \cite{AS,AS08}, relating the zeta function of a member of the family to the $L$-function of an exponential sum.  Our main theorem then follows from a result of Dwork \cite{uniqueness} on the uniqueness of the Frobenius structure on the differential equation and the fact that the Picard--Fuchs equations for the holomorphic forms of $X_{A, \psi}$ and $X_{B, \psi}$ coincide. 

Theorem~\ref{thm:sameweights} overlaps work of Miyatani \cite[Theorem 3.7]{Miyatani}.  In our notation, his theorem states that if $X_{A,\psi}$ is an invertible pencil, $q$ satisfies certain divisibility conditions depending on $A$, and $\psi\in \F_q^\times$ is such that $X_{A,\psi}$ is smooth and $\psi^{d^T} \neq 1$, then $P_{X_{A, \psi}}(T)$ has a factor in $ \overline{\Q}[T]$ that depends only on $q$ and the dual weights $(q_i)_i$.  In particular, if $A$ and $B$ have the same dual weights, the zeta functions of $X_{A,\psi}$ and $X_{B,\psi}$ (for $\psi$ satisfying these conditions) will have a common factor in $\overline{\Q}[T]$.  His factor \cite[(2.4), Remark 3.8(i)]{Miyatani} divides the common factor appearing in Theorem~\ref{thm:sameweights}.  He uses finite-field versions of Gauss sums together with a combinatorial argument.

To compare these two theorems, we observe that Theorem~\ref{thm:sameweights} provides slightly more information about the common factor and places fewer restrictions on $q$: for arithmetic applications, it is essential for the result that it hold without congruence conditions on $q$.  Our techniques are different, and are ruled by the powerful governing principle that factors of the zeta function are organized by Picard--Fuchs differential equations.  For example, our method could extend to pencils for which the associated differential equation may not be hypergeometric.

\subsection{Implications} \ Theorem~\ref{thm:sameweights} relates the zeta functions of many interesting Calabi--Yau varieties: for example, the dual weights are the same for any degree $n+1$ invertible pencil composed of Fermats and loops.  For specificity, we compare the zeta functions of the Dwork pencil $\Fsf_n$ and the generalized Klein--Mukai family $\Fsf_1 \Lsf_n$, defined by the pencil
\begin{equation} \label{eqn:KM}
\Fsf_1 \Lsf_n: x_0^nx_1+\dots+x_{n-1}^nx_0 + x_n^{n+1} - (n+1)\psi x_0x_1\cdots x_n = 0.
\end{equation}
The pencil takes its name from Klein's quartic curve, whose group of orientation-preserving automorphisms is isomorphic to the simple group of order 168, and the member of the family $\Fsf_1\Lsf_{3}$ at $\psi=0$, which appears as an extremal example during Mukai's classification of finite groups of automorphisms of K3 surfaces that preserve a holomorphic form (cf. \cite{eightfold, mukai, OZ}). In this setting, we give a concrete proof of Theorem~\ref{thm:sameweights}.  

We also consider a collection of five invertible pencils $\diamond$ of K3 surfaces in $\PP^4$, including $\Fsf_4$ and $\Fsf_1 \Lsf_3$.  The other three pencils, $\Fsf_2 \Lsf_2$, $\Lsf_2 \Lsf_2$, and $\Lsf_4$, also have only Fermats and loops as atomic types; all five are described by matrices with the same dual weights (see Table~\ref{table:5families} for defining polynomials).  Let $\mathrm{H}$ be the Greene--Plesser mirror family of quartics in $\mathbb{P}^3$, which is obtained by taking the fiberwise quotient of $\Fsf_4$ by $(\mathbb{Z}/4\mathbb{Z})^2$ and resolving singularities.
A computation described by Kadir \cite[Chapter 6]{kadir} shows that for odd primes and $\psi \in \F_q$ such that $\psi^4 \neq 1$ (that is, such that $\mathrm{H}_\psi$ is smooth),
\begin{equation} 
Z(\mathrm{H}_\psi,T)=\frac{1}{(1-T)(1-qT)^{19}(1-q^2T)R_{\psi}(T)}. 
\end{equation}
This calculation combined with Theorem~\ref{thm:sameweights} and properties of K3 surfaces yields the following corollary, exemplifying arithmetic mirror symmetry in these cases.

\begin{cor} \label{cor:yupyup}
Let $\diamond \in \{\Fsf_4,\Fsf_1 \Lsf_3,\Fsf_2 \Lsf_2,\Lsf_2 \Lsf_2,\Lsf_4\}$.  Then there exists $r_0 \geq 1$ such that for all $q=p^r$ with $r_0 \mid r$ and $p \neq 2,5,7$ and all $\psi \in \F_q$ with $\psi^4 \neq 1$, we have 
\[ Z(X_{\diamond,\psi}/\F_{q^r},T)=Z(\mathrm{H}_{\psi}/\F_{q^r},T). \]
\end{cor}

\noindent Accordingly, we could say that the zeta functions $Z(X_{\diamond,\psi}/\F_q,T)$ and $Z(\mathrm{H}_{\psi}/\F_q,T)$ are \defi{potentially equal}---i.e., they are equal after a finite extension of $\F_q$.  (The explicit value of $r_0$ in Corollary \ref{cor:yupyup} will be computed in future work \cite{paperB}.)

Finally, we remark on a simple relationship between the numbers of points of members of alternate mirror families over $\F_q$, reminiscent of the \emph{strong arithmetic mirror symmetry} studied by Fu--Wan \cite{fw}, Wan \cite{wan}, and Magyar--Whitcher \cite{mw}.

\begin{cor} \label{cor:corsmoth}
Let $X_{A,\psi}$ and $X_{B,\psi}$ be invertible pencils of Calabi--Yau $(n-1)$-folds in $\PP^n$ such that $A,B$ have the same dual weights.  Then for all $\psi \in \mathbb{F}_q$, 
\[\#X_{A,\psi}(\mathbb{F}_q) \equiv \#X_{B,\psi}(\mathbb{F}_q) \pmod{q}.\]
\end{cor}

Corollary \ref{cor:corsmoth} is slightly more general than Theorem \ref{thm:sameweights}---there is no hypothesis on the characteristic or on the smoothness of the fiber---but it arrives at a weaker conclusion.

\subsection{Plan of paper} 

In section 2, we introduce our cohomological setup.  In section 3, we consider first the generalized Klein--Mukai family as a warmup to the main theorem, giving a detailed treatment in this case.  In section 4, we prove the main result by recasting a result of G\"ahrs \cite{Gahrs} on Picard--Fuchs equations in hypergeometric terms, study the invariance under symmetry of the middle cohomology, and then apply Dwork cohomology.  To conclude, in section 5, we specialize to the case of K3 surfaces and give some further details for several pencils of particular interest.

\subsection{Acknowledgements}

The authors heartily thank Marco Aldi, Amanda Francis, Xenia de la Ossa, Andrija Peruni\v{c}i\'c, and Noriko Yui for many interesting discussions, as well as Alan Adolphson, Remke Kloosterman, Yang Liping, Fernando Rodriguez--Villegas, Duco van Straten, and the anonymous referee for helpful comments.  They thank the American Institute of Mathematics and its SQuaRE program, the Banff International Research Station, the Clay Mathematics Institute, MATRIX in Australia, and SageMath for facilitating their work together. Doran was supported by  NSERC and the Campobassi Professorship at the University of Maryland. Kelly acknowledges that this material is based upon work supported by the NSF under Award No.\ DMS-1401446 and the EPSRC under EP/N004922/1.  Voight was supported by an NSF CAREER Award (DMS-1151047). 

\section{Cohomological Setup} \label{sec:nondegn}

We begin in this section by setting up notation and establishing a few basic results.  In the cohomology theory of Dwork, following the approach  for related exponential sums as developed by Adolphson--Sperber \cite{AS,AS08}, we will define cohomology spaces endowed with a Frobenius operator with the property that the middle-dimensional primitive factor of the zeta function is realized as the characteristic polynomial of the Frobenius operator acting on non-vanishing cohomology.  
We refer to the work of Adolphson--Sperber for further reference and to Sperber--Voight \cite{SV} for an algorithmic framing.

Throughout the paper, let $\F_q$ be a finite field with $q$ elements and characteristic $p$, with $q=p^a$. Let $\overline{\F}_q$ be an algebraic closure of $\F_q$.  

\subsection{Nondegeneracy and convenience}

Let $F(x)=F(x_0, \dots, x_n) \in \F_q[x_0,\dots,x_n]$ be a nonconstant homogeneous polynomial, so that the vanishing of $F(x)$ defines a projective hypersurface $X \subseteq \PP^{n}_{\F_q}$.  Using multi-index notation, we write 
\[ F(x)=\sum_{\nu \in \Z_{\geq 0}^{n+1}} a_\nu x^{\nu} \]
and $\left|\nu\right|=\sum_{i=0}^{n+1} \nu_i$.  Let $\supp F = \{\nu \in \Z_{\geq 0}^{n+1} : a_\nu \neq 0\}$.  Let $\Delta$ be the convex hull of $\supp F$ and let $\Delta_\infty(F)$ be the convex hull of $\Delta \cup \{(0,\dots,0)\}$ in $\R^{n+1}$.   For a face $\tau \subseteq \Delta$, let 
\[ F|_\tau = \sum_{\nu \in \tau} a_\nu x^\nu. \]

\begin{defn} \label{def:nondeg}
We say $F$ is \defi{nondegenerate} (\defi{with respect to its Newton polyhedron $\Delta $}) if for all faces  $\tau \subseteq \Delta$,  (including $\tau=\Delta$), the system of equations
\begin{equation} \label{eqn:Deltanondeg} 
F|_\tau = \frac{\partial F|_\tau}{\partial x_0} = \dots = \frac{\partial F|_\tau}{\partial x_n} = 0
\end{equation}
has no solutions in $\overline{\F}_q^{\times (n+1)}$.  
\end{defn}

In this case, with $F$ homogeneous, the definition employed by Adolphson and Sperber, that $F$ is \defi{nondegenerate} (\defi{with respect to  $\Delta_{\infty}(F) $}) requires that the system of equations
\begin{equation}  
 \frac{\partial F|_\tau}{\partial x_0} = \dots = \frac{\partial F|_\tau}{\partial x_n} = 0
\end{equation}
has no solutions in $\overline{\F}_q^{\times (n+1)}$ for every face  $\tau \subseteq \Delta$,  (including $\tau=\Delta$). Note that when the characteristic $p$ does not divide the degree of $F$, the Euler relation ensures the two definitions are equivalent.
Finally, we observe if $w$ is a new variable and we consider the form $wF$ then $wF$ is nondegenerate with respect to $\Delta_{\infty}(wF)$ if and only if $F$ is nondegenerate with respect to its Newton polyhedron $\Delta$.

In the calculations below we will make use of a certain positioning of coordinates.  For a subset $J \subseteq \{x_0,\dots,x_n\}$ of variables, we let $F_{\cancel{J}}$ be the polynomial obtained from $F$ by setting the variables in $J$ equal to zero.

\begin{defn}
We say that $F$ is \defi{convenient} with respect to a subset $S \subseteq \{x_0,\dots,x_n\}$ provided that for all subsets $J \subseteq S$, we have
\[ \dim \Delta_\infty(F_{\cancel{J}}) = \dim \Delta_\infty(F)-\# J. \]
\end{defn}

\subsection{Dwork cohomology}

Let $\G_m$ be the multiplicative torus (so $\G_m(\F_q)=\F_q^\times$) and fix a nontrivial additive character $\Theta:\F_q \to \C^\times$ of $\F_q$. Denote by $\Tr_{\F_{q^r}/\F_q} : \F_{q^r} \to \F_q$ the field trace. We will effectively study the important middle dimensional factor of the zeta function by considering an appropriate exponential sum on $\G_m^s \times \A^{n+1-s}$ and treating toric and affine variables somewhat differently.   
For $r \in \Z_{\geq 1}$, define
\[ S_r(F, \G_m^s \times \A^{n+1-s}) \colonequals  \sum_{x \in (\G_m^s \times \A^{n+1-s})(\F_{q^r})} \Theta \circ \Tr_{\F_{q^r}/\F_q} F(x), \]
where the sum runs over all $n+1$-tuples $x=(x_0, \ldots, x_n)$ where $x_0,\dots,x_{s-1} \in \F_{q^r}^\times$ and $x_s,\dots,x_n \in \F_{q^r}$.  Consider the $L$-function of the exponential sum associated to $F$ defined by
\[ L(F,\G_m^s \times \A^{n+1-s},T) \colonequals  \exp\left(\sum_{r=1}^{\infty} S_r \frac{T^r}{r}\right). \]
Then $L(F,\G_m^s \times \A^{n+1-s},T) \in \Q(\zeta_p)(T)$ is a rational function in $T$ with coefficients in the cyclotomic field $\Q(\zeta_p)$, where $\zeta_p$ is a primitive $p$th root of unity.

\begin{thm}[{\cite[Theorem 2.9, Corollary 2.19]{AS}}] \label{thm:AS} 
If $F$ is nondegenerate and convenient with respect to $S=\{x_{s+1},\dots,x_n\}$, and $\dim \Delta_\infty(F)=n+1$, then the $L$-function
\[ L(F, \G_m^s \times \A^{n+1-s},T)^{(-1)^{n+1}} \in \Q(\zeta_p)[T] \]
is a polynomial in $T$ with coefficients in $\Q(\zeta_p)$ of degree given explicitly in terms of the volumes $\vol \Delta_\infty(F_{\cancel{J}})$ for ${J \subseteq S}$.  
\end{thm}

This theorem also gives information about the $p$-adic size of the reciprocal zeros of $L(T)^{(-1)^{n+1}}$.  

We now proceed to relate the $L$-function of such an exponential sum to the zeta function of the corresponding hypersurface. In general, we write
\begin{equation}  \label{eqn:Zingen}
Z(X,T) \colonequals \exp\left(\sum_{r=1}^{\infty} \#X(\F_{q^r})\frac{T^r}{r}\right) =  \frac{P(T)^{(-1)^n}}{(1-T)\cdots (1-q^{n-1}T)} 
\end{equation}
with $P(T) \in \Q(T)$.  If $X$ is smooth and $F$ has degree $d$, then $P(T)$ is a polynomial of degree 
\begin{equation} \label{eqn:degP}
\deg P = \frac{d-1}{d}((d-1)^n+(-1)^{n+1}),
\end{equation}
representing the characteristic polynomial of Frobenius acting on the primitive middle-dimensional cohomology of $X$. 
 Let $Y \subseteq \A^{n+1}$ be the affine hypersurface defined by the vanishing of $F$, the cone over $X$.  Let $w$ be a new variable.  A standard argument with character sums shows that
\begin{equation} \label{eqn:numofzeros}
S_r(wF,\A^{n+2}) = q^r \#Y(\F_{q^r}). 
\end{equation}
Therefore
\[ L(wF,\A^{n+2},T) = Z(Y,qT). \]
On the other hand, one has
\[ Z(Y,T)=\frac{Z(X,qT)}{Z(X,T)(1-T)}. \]
So putting these together we have
\begin{equation} \label{eqn:LwF}
L(wF,\A^{n+2},T)=\frac{Z(X,q^2T)}{Z(X,qT)(1-qT)}. 
\end{equation}

By combining Equations~\eqref{eqn:Zingen} and~\eqref{eqn:LwF}, we have
\begin{equation} 
L(wF,\A^{n+2},T) = \left(\frac{P(qT)}{P(q^2T)}\right)^{(-1)^{n+1}} \frac{1}{1-q^{n+1}T}. 
\end{equation}
Finally, splitting the domain for the variable $w$ as $\A^1=\G_m \cup \{0\}$, we obtain
\begin{equation} \label{eqn:LPqT}
L(wF,\G_m \times \A^{n+1},T)^{(-1)^{n+1}}=\frac{P(qT)}{P(q^2T)}. 
\end{equation}

In the special case where $F$ is nondegenerate with respect to $\Delta_\infty(F)$ and convenient with respect to $\{x_0,\dots,x_n\}$, Theorem \ref{thm:AS} applies.   Under these hypotheses, Adolphson\textendash Sperber \cite[Section 6]{AS} prove the following: there exists a $p$-adic cohomology complex $\Omega^{\bullet}$ such that the trace formula
\begin{equation} \label{eqn:KMtrace}
L(wF,\G_m \times \A^{n+1},T) = \prod_{i=0}^{n+2} \det(1-\Frob T \mid H^i(\Omega^{\bullet}))^{(-1)^{i+1}} 
\end{equation}
holds, the cohomology groups $H^i(\Omega^{\bullet})$ vanish for $i=0,\dots,n$, we have
\begin{equation} \label{eqn:KMident}
 \Frob \mid H^{n+1}(\Omega^{\bullet})= q \Frob \mid H^{n+2}(\Omega^{\bullet}),
 \end{equation}
and finally
\begin{equation} \label{eqn:PqT}
P(qT)=\det(1-\Frob T \mid H^{n+2}(\Omega^{\bullet})). 
\end{equation}
For more details, see also Adolphson\textendash Sperber \cite[Corollary 6.23]{AS08} and Sperber\textendash Voight \cite[Section 1 and pages 31-32]{SV}.
In particular, the formula \eqref{eqn:PqT} gives a fairly direct way to compute $P(T)$ in the case of the Dwork family of hypersurfaces, since the defining polynomial $F$ is convenient with respect to the full set of variables $\{x_0,\dots,x_n\}$.

\subsection{Unit roots} \label{sec:unitroot}

For convenience, we conclude this section by recalling the relationship between Hodge numbers and the $p$-adic absolute values of the reciprocal zeros and poles of the zeta function.  

The following is a consequence of the Katz conjecture proved in full generality by  Mazur \cite{Mazur}. In the present context, it follows directly from Adolphson--Sperber \cite[Theorem 3.10]{AS}.

\begin{prop} \label{prop:Newtonoverhodge}
The Newton polygon of $P_{\diamond,\psi,q}(T)$ lies over the Hodge polygon of middle-dimensional primitive cohomology.
\end{prop}

We now apply this to our invertible pencils, as defined in \eqref{eqn:invertpencil}.  In particular, we have $X_{A,\psi}$ a smooth projective hypersurface in $\PP^n$ defined by a polynomial $F_{A,\psi}$ of degree $n+1$, so $X_{A,\psi}$ is a Calabi--Yau variety of dimension $n-1$.  By a standard calculation, the first Hodge number of $X_{A,\psi}$ is $h^{0,n-1}=1$.  Therefore the Hodge polygon of middle-dimensional primitive cohomology starts with a segment of slope zero having length $1$.  By Proposition \ref{prop:Newtonoverhodge}, there is at most one reciprocal root of the polynomial $P_{\diamond,\psi,q}(T)$ that is a $p$-adic unit: we call this reciprocal root when it occurs a \defi{unit root}. 

\begin{exm}
If $n=3$ and $\deg F=4$, and $X$ is smooth, then $X$ is a quartic K3 surface, and so the Newton polygon of $P_{\diamond,\psi,q}(T)$ lies over the Hodge polygon (i.e., the Newton polygon of  $(1-T)(1-qT)^{19} (1-q^2 T)$).
\end{exm}

There is a polynomial defined over $\mathbb{F}_p$ depending on $A$, called the \defi{Hasse invariant}, with the property that $H_A(\psi) \neq 0$ for a smooth fiber $\psi \in \F_q^\times$ if and only if there is a unique unit root. In this case,  we call $X_{A,\psi}$ \defi{ordinary}, otherwise we say $X_{A,\psi}$ is \defi{supersingular}. The polynomial $H_A$ is nonzero as the monomial $x_0x_1\dots x_n$ appears in $F_{A,\psi}$ \cite[(1.9), Example 1]{AS16} (in their notation, we have $\mu=0$).  Therefore, the ordinary sublocus of $\PP^1 \smallsetminus \{0,1,\infty\}$ is a nonempty Zariski open subset.  This unit root has seen much study: for the Dwork family, it was investigated by Jeng-Daw Yu \cite{yu}, and in this generality by Adolphson--Sperber \cite[Proposition 1.8]{AS16} (see also work of Miyatani \cite{Miyatani}).  

These $p$-adic estimates can be seen explicitly in Dwork cohomology, as follows.  By \eqref{eqn:PqT}, for the hypersurface $X_{A,\psi}$ we are interested in the action of $q^{-1}\Frob$ on the cohomology group $H^{n+2}(\Omega^{\bullet})$.

\begin{lem} \label{lem:qinFrob}
The operator $q^{-1}\Frob$ acting on $H^{n+2}(\Omega^{\bullet})$ reduced modulo $p$ has rank at most $1$ and has rank exactly $1$ if and only if $X_{A,\psi}$ is ordinary.  
\end{lem}

\begin{proof}
The cohomology group $H^{n+2}(\Omega^{\bullet})$ has a basis of monomials $\{(\gamma w)^{\left|\nu\right|/d} x^{\nu}\}_{\nu}$ where $d \mid \left|\nu\right|$ and $\gamma \in \Z_p[\zeta_p]$ is a uniformizer \cite[Section 5, ``Modifications: projective varieties'']{SV}.   Let $A^0=(a^0_{\mu\nu})_{\mu,\nu}$ be the matrix of the $p$-Frobenius on this basis.  Then \cite[Proposition 3.9]{AS}
\begin{equation} \label{eqn:a0munu}
\ord_p a^0_{\mu\nu} \geq \frac{\left|\mu\right|}{d}.
\end{equation}
Let $A=(a_{\mu\nu})_{\mu,\nu}$ be the matrix of the $q$-Frobenius $\Frob$.  By \eqref{eqn:PqT}, we have $P(T)=\det(1-q^{-1}AT)$, and \eqref{eqn:a0munu} implies (as in Adolphson--Sperber \cite[Proof of Theorem 3.10]{AS})
\begin{equation} \label{eqn:amunu}
\ord_q (q^{-1}a_{\mu\nu}) \geq \frac{\left|\mu\right|}{d} - 1. 
\end{equation}
By the Calabi--Yau condition, the unique monomial with $\left|\mu\right|/d=1$ is 
\begin{equation} \label{eqn:defofomega0}
\omega_0 \colonequals \gamma wx_0\cdots x_n,
\end{equation} 
so \eqref{eqn:amunu} implies that the matrix $q^{-1} A$ has at most one nonzero column modulo $p$, so its reduced rank is at most $1$; and this rank is equal to $1$ if and only if its characteristic polynomial has a nonzero root, i.e., if and only if $X_{A,\psi}$ is ordinary, by definition.  Moreover, the rank is equal to $1$ if and only if 
\begin{equation} \label{eqn:anunu}
a_{\nu\nu} \not\equiv 0 \psmod{p}
\end{equation}
where $\nu=(1,1,\dots,1)$ corresponds to $\omega_0$.
\end{proof}

\section{Generalized Klein--Mukai Family}

As a warm-up to the main theorem, we now consider in detail the generalized Klein--Mukai family $\Fsf_1\Lsf_{n}$ of Calabi--Yau $n$-folds.  We give a proof of the existence of a common factor---realizing these as alternate mirrors, from the point of view of $p$-adic cohomology.  Since it is of particular interest, and has rather special features, along the way we provide further explicit details about this family.

\subsection{Basic properties}

For $n \geq 1$, let 
\begin{equation} \label{eqn:Fx0nkm} 
F(x)=F_\psi(x) \colonequals x_0^nx_1+\dots+x_{n-1}^nx_0 + x_n^{n+1} - (n+1)\psi x_0x_1\cdots x_n.
\end{equation}
and define $X_\psi \subseteq \PP^n$ to be the \defi{generalized Klein--Mukai} family of hypersurfaces over $\Z$ defined by the vanishing of $F_\psi$.  The polynomial \eqref{eqn:Fx0nkm} of degree $n+1$ in $n+1$ variables may be described as consisting of a single Fermat term together with a single loop of length $n$, so we will also refer to it by the symbol $\Fsf_1 \Lsf_n$.  

Throughout, let $m \colonequals  n^n+(-1)^{n+1}$.  Note $(n+1) \mid m$. Let $k$ be a field and $\zeta \in k$ a primitive $m$th root of unity.

\begin{lem}
Suppose $p \nmid m$.  For $\psi \neq 0$, the group
$$
G(k)=\{ \lambda = (\lambda_i)_i \in \mathbb{G}_m^{n+1}(k) : F_{\psi}(\lambda x) = F_{\psi}(x) \}
$$
is a cyclic group of order $m$, generated by $z=(\zeta, \zeta^{-n}, \zeta^{n^2}, \ldots, \zeta^{(-n)^{n-1}}, \zeta^{(-1)^n m /(n+1)})$. The subgroup acting trivially on $X_{\psi}$ is  cyclic of order $n+1$, and the quotient acting faithfully on $X_{\psi}$ is generated by $z^{n+1}$.  
\end{lem}

\begin{proof}
This statement follows from a direct computation.  \end{proof}

\begin{lem}
Suppose $p \nmid m$.  Then for all $\psi \in \F_q$ such that $\psi^{n+1} \neq 1$, the hypersurface defined by $F_{\psi}(x)$ is smooth, nondegenerate, and convenient with respect to $\{x_n\}$.
\end{lem}

\begin{proof}
The statement on convenience is immediate.

We begin with the full face $\Delta$, where nondegeneracy (using the Euler relation) is equivalent to smoothness.  We compute for $i=0,\dots,n-1$ that
\begin{equation} \label{eqn:xideriv}
x_i\displaystyle{\frac{\partial F}{\partial x_i}} = x_{i-1}^n x_i + nx_i^n x_{i+1} - (n+1)\psi x_0 x_1 \cdots x_n 
\end{equation}
with indices taken modulo $n$, and 
\begin{equation} \label{eqn:xnderiv}
x_n\displaystyle{\frac{\partial F}{\partial x_n}} = (n+1)x_n^{n+1} - (n+1)\psi x_0 x_1 \cdots x_n.
\end{equation}
Setting these partials to zero and subtracting \eqref{eqn:xnderiv} from \eqref{eqn:xideriv}, we obtain the $n \times (n+1)$-matrix equation
\begin{equation} \label{eqn:mateqn}
\begin{pmatrix}
1 & n & 0 & \cdots & 0 & 0 & -(n+1) \\
0 & 1 & n & \cdots & 0 & 0 & -(n+1) \\
\vdots & \vdots & \vdots & \ddots & \vdots & \vdots & \vdots \\
0 & 0 & 0 & \cdots & 1 & n & -(n+1) \\
n & 0 & 0 & \cdots & 0 & 1 & -(n+1) \\
\end{pmatrix}
\begin{pmatrix}
x_0^n x_1 \\ x_1^n x_2 \\ x_2^n x_3 \\ \vdots \\ x_{n-1}^n x_0 \\ x_n^{n+1}
\end{pmatrix} = 0. 
\end{equation}
The absolute value of the determinant of the left $n \times n$ block of the matrix in \eqref{eqn:mateqn} is $m=n^n+(-1)^{n+1}$, so by our assumption on $p$ the full matrix has rank $n$ over $\F_q$.  By homogeneity, the vector $(1,\dots,1)^t$ therefore generates the kernel of the full matrix; the solution vector lies in this kernel, so we conclude
\[ x_0^nx_1=x_1^nx_2=x_2^nx_3=\cdots=x_{n-1}^nx_0 = x_n^{n+1}. \]
Since $x \in \overline{\F}_q^{\times (n+1)}$, by scaling we may assume $x_n=1$.  Thus $x_{i-1}^n x_{i} = 1$ for $i= 1,\ldots,n-1$; taking the product of these gives $(x_0\cdots x_{n-1})^{n+1}=1$.  Since $\psi x_0\cdots x_n = 1$ as well, we conclude $\psi^{n+1}=1$; and these are precisely the excluded values.

Now suppose that $\tau \subsetneq \Delta$ is a proper face of $\Delta$. Then clearly $(1,1,\ldots,1)$ does not belong to $\tau$.   If $\tau$ contains $(0,\dots,0,n+1)$, then by restricting \eqref{eqn:xnderiv} to $\tau$, we see that a zero of $x_n\displaystyle{\frac{\partial F|_\tau}{\partial x_n}}= (n+1)x_n^{n+1}$ must have $x_n=0$, so we may assume $\tau$ does not contain the vertex $(0,\dots,0,n+1)$. If $\tau$ does not contain all of the ${x_{i-1}^n x_i}$ then at least one variable $x_i$ with $i \in \{0,\dots,n-1 \}$ appears in only one monomial of $F|_{\tau}$, so that a zero of $\displaystyle{\frac{\partial F|_\tau}{\partial x_i}}$ must have a zero coordinate.  
The only other possibility for a face $\tau$ is the one  corresponding to letting $x_n=0$ in $F$, i.e., the loop equation itself.  Writing the equations \eqref{eqn:xideriv} with $x_n=0$ in matrix form yields the left $n \times n$-block of the matrix in \eqref{eqn:mateqn}; but now, since $p \nmid m$, a point of nondegeneracy must be $(0,\dots,0)$, proving the nondegeneracy of $F$.
\end{proof}

To overcome the fact that the generalized Klein--Mukai pencil is only convenient with respect to $\{x_n\}$ (as opposed to the case of the Dwork pencil, which is convenient with respect to the full set of variables $\{x_0,\dots,x_n\}$), we prove the following lemma.

\begin{lem} \label{lem:pn1km}
We have
\[ L(wF, \G_m \times \A^{n+1},T) = L(wF,\G_m^{n+1} \times \A^1,T). \]
\end{lem}

The point of this combinatorial lemma is that one obtains the same value of the exponential sum when changing affine coordinates to toric coordinates, so that Theorem \ref{thm:AS} applies.

\begin{proof}
Let $S=\{0,\dots,n-1\}$ and $J \subseteq S$ with $J^c=S - J$.  Write $\mathbb{A}^{J^c} \subseteq \A^{n+1}$ for the linear subspace defined by the vanishing of $x_i=0$ for $i \in J$.
Recall that $F_{\cancel{J}}(x) \in \F_q[x_i]_{i \in J^c}$ is the polynomial obtained from $F(x)$ by setting the variables in $J$ equal to zero.  

Let $r \in \Z_{\geq 0}$.  A standard inclusion-exclusion argument gives
\begin{equation} \label{eqn:SRsum}
S_r(wF,\mathbb{G}_m \times \mathbb{A}^{n+1})
= S_r(wF,\mathbb{G}_m^{n+1} \times \mathbb{A}^1) 
+ \sum_{\substack{J \subseteq S \\ J \neq \emptyset}} (-1)^{\#J+1} S_r(wF_{\cancel{J}},\mathbb{G}_m \times \mathbb{A}^{J^c} \times \mathbb{A}^1) 
\end{equation}
We claim, in fact, that every summand on the right-hand side of \eqref{eqn:SRsum} is zero; that is, if $J \neq \emptyset$, that
\begin{equation} 
S_r(wF_{\cancel{J}},\mathbb{G}_m \times \mathbb{A}^{J^c} \times \mathbb{A}^1)=0. 
\end{equation}

To this end, suppose that $J \neq \emptyset$; then at least one coordinate is sent to zero in $F_{\cancel{J}}(x)$, and the deforming monomial $x_0\cdots x_n$ is set to zero.  

First suppose that $\#J^c \leq 1$.  Then $F_{\cancel{J}}(x)=x_n^{n+1}$, and 
\[ S_r(wF_{\cancel{J}},\mathbb{G}_m \times \mathbb{A}^{J^c} \times \mathbb{A}^1) = q^{rt} S_r(wx_n^{n+1},\mathbb{G}_m \times \mathbb{A}^1) \]
with $t = \#J^c$.  We then compute that 
\[ S_r(wx_n^{n+1},\mathbb{G}_m \times \mathbb{A}^1) = S_r(wx_n^{n+1},\mathbb{A}^2) - S_r(0,\mathbb{A}^1) = q^{r}-q^{r}=0 \]
by \eqref{eqn:numofzeros}.  

So suppose $\#J^c \geq 2$.  If the loop vanishes, we again have $F_{\cancel{J}}(x)=x_n^{n+1}$ and we are back in the previous case.  So we may assume that at least one of the surviving coordinates appearing linearly: there exists $j \in S$ such that $j-1,j \in J^c$ hence 
\[ F_{\cancel{J}}(x) = F_{\cancel{J}'}(x) + x_{j-1}^n x_j \]
with $J'=J \cup \{j\}$.  But then $(J')^c \cup \{j\}=J^c$, so
\begin{equation} \label{eqn:giantsum} 
\begin{aligned}
&S_r(wF_{\cancel{J}},\mathbb{G}_m \times \mathbb{A}^{J^c} \times \mathbb{A}^1)
\\
&\qquad\qquad= \sum_{w \in \F_{q^r}^\times} 
\sum_{x \in \F_{q^r}^{(J')^c}} 
(\Theta \circ \Tr_{\F_q^r/\F_q})(wF_{\cancel{J}'}(x))
\sum_{x_j \in \F_{q^r}} (\Theta \circ \Tr_{\F_{q^r}/\F_q})(wx_{j-1}^n x_j). 
\end{aligned}
\end{equation}
Summing the innermost sum on the right side of \eqref{eqn:giantsum} over $x_j \in \F_{q^r}$  counts  with multiplicity $q^r$ the  number of zeros of $wx_{j-1}^n$ with $w \in \F_{q^r}^\times$, where $x_{j-1} \in \F_{q^r}$ is fixed.  If $x_{j-1} \neq 0$, then there are no such zeros and the inner sum is zero.  Therefore, letting $J''=J \cup \{j-1,j\}$ (with indices taken modulo $n$),
\begin{equation} \label{eqn:JtoJpp}
S_r(wF_{\cancel{J}},\mathbb{G}_m \times \mathbb{A}^{J^c} \times \mathbb{A}^1)
= q^{r} S_r(wF_{\cancel{J''}},\mathbb{G}_m \times \mathbb{A}^{(J'')^c} \times \mathbb{A}^1).
\end{equation}
Replacing $J$ by $J''$, we iterate the argument and reduce to the case where $\#J^c \leq 1$, completing the proof.
\end{proof}

With Lemma \ref{lem:pn1km} in hand, we can now conclude as with the Dwork family: since $F(x)$ is nondegenerate and convenient with respect to $S=\{x_n\}$, the proof of Theorem \ref{thm:AS} yields a $p$-adic cohomology complex $\Omega^{\bullet}$ such that as in \eqref{eqn:PqT} we have
\[ P(qT)=\det(1-\Frob T \mid H^{n+2}(\Omega^{\bullet})) \]
By \eqref{eqn:degP}, we find that $P(T)$ is a polynomial of degree 
\[ \deg P = \frac{nm}{n+1}=\frac{n^{n+1}+(-1)^{n+1}n}{n+1}. \]
In this way, we have shown that the characteristic polynomial of Frobenius acting on middle-dimensional cohomology for the Klein--Mukai family can be computed by its action on a cohomology group.

\subsection{Common factors}

We now identify factors in common for the Dwork and generalized Klein--Mukai pencils $\diamond \in \{\Fsf_{n+1},\Fsf_1\Lsf_n\}$.  

The Picard--Fuchs equation defined by the action of the operator $\displaystyle{\psi\frac{\partial}{\partial \psi}}$ on the unique nonvanishing holomorphic differential has rank $n$ in both cases.  After a change of variables, this Picard--Fuchs equation is the differential equation satisfied by the classical hypergeometric function
\begin{equation} \label{eqn:hypnnm}
{} \psi^{-1}\, {}_nF_{n-1}\left(\begin{array}{c}\frac{1}{n+1},\frac{2}{n+1},\ldots,\frac{n}{n+1} \\1, \ldots, 1 \end{array}; \psi^{-1/(n+1)}\right)
\end{equation} \cite[Corollary 2.3.8.1]{Kat72}. 

Let $S$ be the set of variables of $F$ appearing in the Fermat (diagonal form) piece of the defining polynomial $F$ in either case.  Then $F$ is convenient with respect to $S$.  Suppose $\psi \in \F_q$ is such that $F_{\psi}(x)$ is nondegenerate with respect to $\Delta_{\infty}(F)$.  Therefore, we have a $p$-adic complex $\Omega^{\bullet}$ such that \eqref{eqn:KMtrace}--\eqref{eqn:PqT} hold.

We prove that for each fiber, the zeta functions in these two families have middle-dimension\-al cohomology with a common factor of degree $n$ determined by action of the connection on the $\epsilon(\partial/\partial \psi)$-stable subspace containing the unique holomorphic nonvanishing differential $n$-form.  In both cases, the monomial $wx_0x_1\cdots x_n \in \Omega^{n+2}$ corresponds to this $n$-form.  For $q=p^r$, let $\Q_q$ be the unramified extension of $\Q_p$ of degree $r$.

\begin{prop} \label{prop:uniqueoffrob}
If $p \nmid (n+1)\dt$ and $\psi \in \F_q^\times$ is a smooth, nondegenerate fiber, then the polynomials $P_{\diamond,\psi}(T)$ where $\diamond \in \{\Fsf_{n+1}, \Fsf_1\Lsf_n \} $ have a common factor $R_{\psi}(T) \in \Q_q[T]$ of degree $n$. 
\end{prop}

\begin{proof}
Viewed over a ring with derivation $\partial/\partial \psi$, for all $i$ the cohomology $H^i(\Omega^{\bullet})$ has an action by the connection 
\[ \epsilon\left(\frac{\partial}{\partial \psi}\right) = \frac{\partial}{\partial \psi} - (n+1)\gamma_0 \psi wx_0 x_1 \cdots x_n \]
where $\gamma_0$ is an appropriate $p$-adic constant.  The monomial $wx_0 x_1 \cdots x_n$ then spans an $\epsilon(\partial/\partial \psi)$-stable subspace of $H^{n+2}(\Omega_{\diamond}^{\bullet})$, denoted $\Sigma_{\diamond}$.  In both cases $\diamond \in \{\Fsf_{n+1}, \Fsf_1\Lsf_n \} $, we have a Frobenius map $\Frob_{\diamond}^{\bullet}$ acting as a chain map on the complex $\Omega^{\bullet}$ and stable on $\Sigma_{\diamond}$.  As a consequence, we conclude that
\[ P_{\diamond}(qT)=\det(1-T\Frob_{\diamond} \mid H^{n+2}(\Omega^{\bullet}_{\diamond})) = \det(1-T\Frob_{\diamond} \mid \Sigma_{\diamond}) Q_{\diamond}(T). \]

Let $\Phi_{\diamond}(\psi)$ represent the Frobenius map $\Frob_{\diamond}$ restricted to $\Sigma_{\diamond}$.  We appeal to work of Dwork \cite{padic}.  We find that in the sense of Dwork, there are two Frobenius structures, both of which are \emph{strong Frobenius structures} as a function of the parameter $\psi$ on the hypergeometric differential equation, corresponding to the two values of $\diamond$.  The hypergeometric differential equation (over $\C_p$, or any field of characteristic zero) is irreducible because none of the numerator parameters $\{1/(n+1),\dots,n/(n+1)\}$ differ from the denominator parameter $\{1\}$ by an integer \cite[Corollary 1.2.2]{Beu}.  As a consequence, the hypotheses of a lemma of Dwork \cite[Lemma, p.\ 89--90]{uniqueness} are satisfied, and we have that the two Frobenius structures agree up to a multiplicative constant  $c \in \C_p^\times$; in terms of matrices,
\[ \Phi_{\Fsf_{n+1}}(\psi)=c \Phi_{\Fsf_1\Lsf_n}(\psi). \] 

We now show that $c=1$.  Let $\psi_0 \in \F_q$ be such that $\psi_0^{n+1} \neq 1$.  Then the fiber for each family at $\psi =\psi_0$ satisfies $F_{\diamond, \psi_0}(x) = 0$ and the defining polynomial $wF_{\diamond, \psi_0}(x)$ is nondegenerate.  Let $\widehat{\psi}_0$ be the Teichm\"uller lift of $\psi_0$.  We recall section \ref{sec:unitroot}.  Suppose that $\psi_0$ is an ordinary fiber for both families.  Then
\[ \Tr(\Phi_{\Fsf_{n+1}}(\widehat{\psi}_0)) = c \Tr(\Phi_{\Fsf_1\Lsf_n}(\widehat{\psi}_0)). \]
Without loss of generality, we may assume that $c$ is a $p$-adic integer.  Since the two families have the same Picard--Fuchs differential equation, we obtain $p$-adic analytic formulas for the unique unit root of $\Tr(\Phi_{\Fsf_{n+1}}(\widehat{\psi}_0))$  by Jeng-Daw Yu \cite{yu}, and for the unique unit root of $\Tr(\Phi_{\Fsf_1\Lsf_n}(\widehat{\psi}_0))$ by work of Adolphson--Sperber \cite{AS16} (also proven by Miyatani \cite{Miyatani}).  These formulas are given in terms of the unique holomorphic solution of Picard--Fuchs (at $\infty$) so that the formulas are the same, so the unique unit roots for the two families agree, and this forces $c \equiv 1 \psmod{q}$.  Repeating this argument over all extensions $\F_{q^r}$ with $r \geq 1$, we conclude similarly that $c^r \equiv 1 \psmod{q^r}$.  Taking $r$ coprime to $p$, by binomial expansion we conclude $c=1$ as desired.
\end{proof}

In the next section, we generalize this result and also prove that $R_\psi(T) \in \Q[T]$.
 
\section{Proof of the main result}

We now prove the main theorem in the general setting of families of alternate mirrors.

\subsection{Hypergeometric Picard--Fuchs equations}\label{subsec: PFEquations}

To begin, we study the Picard--Fuchs equation for the holomorphic form of an invertible pencil.  We use the structure of the Picard--Fuchs equation to identify a factor of the zeta function associated to the holomorphic form, establishing a version of our main theorem with coefficients defined over a number field. By work of G\"ahrs \cite{Gahrs, GahrsThesis}, we know that if two invertible pencils have the same dual weights, then their Picard--Fuchs equations are the same. We now state her result and recast it in a hypergeometric setting. 

Let $F_A$ be an invertible polynomial, where $q_i$ are its dual weights and $\dt \colonequals \sum_i q_i$ is the weighted degree of the transposed polynomial $F_{A^T}$.  For each $\psi$, let $\mathcal{H}^n(X_{A,\psi})$ be the de Rham cohomology of the holomorphic $n$-forms on the complement $\PP^n \setminus X_{A, \psi}$, and write the usual holomorphic form on $\PP^n$ as $\Omega_0 = \sum_{i=0}^n (-1)^i x_i \,\dD{x_0}\wedge \ldots \wedge \dD{x_{i-1}} \wedge \dD{x_{i+1}} \wedge \ldots \wedge \dD{x_n}$.  Then one may use the Griffiths residue map $\Res: \mathcal{H}^n(X_{A, \psi}) \rightarrow H^{n-1}(X_{A, \psi},\C)$, whose image is primitive cohomology, to realize the holomorphic form on $X_{A, \psi}$ as $\Res(\Omega_0/F_{A,\psi})$.  Systematically taking derivatives of the holomorphic form establishes the Picard--Fuchs differential equation associated to the holomorphic form that G\"ahrs computes via a combinatorial formulation of the Griffiths-Dwork technique.  We now state her result.

We first define the rational numbers
\begin{equation}\begin{aligned}
\alpha_j &\colonequals  \frac{j}{\dt}, & & \text{ for }j=0,\dots,\dt-1; \\ \beta_{ij} &\colonequals   \frac{j}{q_i}, & & \text{ for }i=0,\dots,n \text{ and }j=0,\ldots, q_i-1.
\end{aligned}\end{equation} 
Consider the multisets (sets allowing possible repetition)
\begin{equation}
\begin{aligned}
\pmb{\alpha} &\colonequals  \left\{ \alpha_j : j=0,\dots,\dt-1\right\}; \\
\pmb{\beta}_i &\colonequals  \left\{ \beta_{ij} : j = 0,\dots,q_i-1\right\}, \quad \pmb{\beta} \colonequals \bigcup_{i=0}^n \pmb{\beta}_i.
\end{aligned}
\end{equation}
The elements of the multiset $\pmb{\alpha}$ have no repetition, so we can think of $\pmb{\alpha}$ as a set.  Take the intersection $I = \pmb{\alpha} \cap \pmb{\beta}$.  Note that all of these sets depend only on the dual weights $q_i$.  Let $\delta = \psi \displaystyle{\frac{d}{d\psi}}$. 

\begin{thm}[G\"ahrs]\label{thm: Gahrs} 
Let $X_{A, \psi}$ be an invertible pencil of Calabi--Yau $(n-1)$-folds determined by the integer matrix $A$, with dual weights $(q_0, \ldots, q_n)$. Then the following statements hold.
\begin{enumalph}
\item The order of the Picard--Fuchs equation for the holomorphic form of the invertible pencil is 
\begin{equation}\label{OrderPFEquation}
   D(q_0,\dots,q_n) : =  \dt - \# I.
\end{equation} 
\item The Picard--Fuchs equation itself is given by
\begin{equation}\label{PFEquation}
    \left(\prod_{i=0}^n q_i^{q_i} \right)\psi^{\dt}  \left( \prod_{\beta_{ij} \in \pmb{\beta} \smallsetminus I} (\delta + \beta_{ij}\dt)\right)  -  \prod_{\alpha_j \in \pmb{\alpha} \smallsetminus I} (\delta - \alpha_j\dt)=0.
\end{equation}
\end{enumalph}
\end{thm}

\begin{proof}
Part (a) is due to G\"ahrs \cite[Theorem 2.8]{GahrsThesis}, and part (b) is a slight reparameterization of variables of a result also due to G\"ahrs \cite[Theorem 6]{Gahrs}.  
\end{proof}

The Picard--Fuchs equation can be written in hypergeometric form.  Indeed, if we change variables with
\begin{equation} 
z \colonequals \biggl(\prod_i q_i^{-q_i }\biggr) \psi^{-\dt}, \qquad \theta \colonequals z \frac{d}{dz} = -(\dt)^{-1} \delta, 
\end{equation}
we may rewrite the Picard--Fuchs equation as 
\begin{equation} \label{eqn: PFhypergeo} 
\prod_{\beta_{ij} \in \pmb{\beta} \smallsetminus I} \left(\theta - \beta_{ij}\right) -  z\prod_{\alpha_j \in \pmb{\alpha} \smallsetminus I} \left(\theta + \alpha_j\right)=0.
\end{equation}
As $\beta_{i0} = 0 \in \pmb{\beta}$ for all $i$, we have $0 \in \pmb{\beta} \smallsetminus I$, hence the Picard--Fuchs equation is a hypergeometric differential equation. In particular, a solution is given by the (generalized) hypergeometric function
\begin{equation}\label{equ:genhypergeometric}
\phantom{i} {}_{D}F_{D-1}\left(\begin{array}{c}\alpha_i \in \pmb{\alpha}\smallsetminus I \\ \beta_{ij} \in \pmb{\beta} \smallsetminus (I \pmb{\cup} \{0\}) \end{array};\, (\textstyle{\prod}_i q_i^{-q_i}) \psi^{-\dt}\right),
\end{equation} 
where $D=D(q_0,\dots,q_n)$ and $I \pmb{\cup} \{0\}$ is the multiset obtained by adjoining $0$ to $I$.

\begin{exm}\label{example: K3QuarticsPF}
Consider a pencil $X_{A, \psi}$ of quartic projective hypersurfaces with dual weights $(1,1,1,1)$.  Then $ \pmb{\alpha} = \{0, \frac{1}{4}, \frac{2}{4}, \frac{3}{4}\}$ and $\pmb{\beta} = \{ 0,0,0,0\}$.  Since $I= \{0\}$, the Picard--Fuchs equation is of the form 
$$
\theta^3 - \lambda \left(\theta + \frac{1}{4}\right)\left(\theta+ \frac{1}{2}\right)\left(\theta+\frac{3}{4}\right) = 0,
$$
which is a hypergeometric differential equation satisfied by the hypergeometric function
\begin{equation}
\phantom{i} {}_3F_{2}\left(\begin{array}{c}\frac{1}{4},\frac{1}{2},\frac{3}{4} \\1, 1 \end{array}; \psi^{-4}\right).
\end{equation} 
\end{exm}

\begin{prop}\label{prop:PFirred}
The Picard--Fuchs equation given in Equation~\eqref{eqn: PFhypergeo} is irreducible.
\end{prop}

\begin{proof}
This differential equation has parameters such that $\alpha_i - \beta_{jk} \not\in \Z$ for all $i,j,k$, for the following reason: the elements of $\pmb{\alpha}$ and $\pmb{\beta}$ are already in $[0,1)$, so two differ by an integer if and only if they are equal; and whenever two coincide, they are taken away by the set $I$ (noting the elements of $\pmb{\alpha}$ are distinct).  Therefore, the differential equation is irreducible \cite[Corollary 1.2.2]{Beu}. 
\end{proof}

\subsection{Group invariance}\label{subsec:GroupInvariance}

In this section, we show that the subspace of cohomology associated to the Picard--Fuchs equation for the holomorphic form is contained in the subspace fixed by the action of a finite group.  This group arises naturally in the context of Berglund--H\"{u}bsch--Krawitz mirror symmetry.  Throughout, we work over $\C$.

We begin by establishing three groups that are useful when studying invertible potentials and prove a result about the invariant pieces of cohomology associated to them. Let $F_A$ be an invertible polynomial.  First, consider the elements of the maximal torus $\G_m^{n+1}$ acting diagonally on $\PP^n$ and leaving the polynomial $F_A$ invariant:
\begin{equation}
\Aut(F_A) \colonequals \{ (\lambda_0, \ldots, \lambda_n) \in \G_m^{n+1}  : F_A(\lambda_i x_i) = F_A(x_i)\} \subseteq \GL_{n+1}(\C).
\end{equation}
Write $A^{-1}=(b_{ij})_{i,j} \in \GL_{n+1}(\Q)$ and for $j=0,\dots,n$ let
\[ \rho_j= (\exp(2\pi i b_{0j}), \ldots, \exp(2\pi i b_{nj})); \] 
then $\rho_0,\dots,\rho_n$ generate $\Aut(F_A)$.  

Next, we consider the subgroup
\begin{equation}
\SL(F_A) \colonequals  \{ (\lambda_0, \ldots, \lambda_n) \in \Aut(F_A)  :  \lambda_0 \cdots\lambda_n = 1 \} = \Aut(F_A) \cap \SL_{n+1}(\C)
\end{equation}
acting invariantly on the holomorphic form, and the subgroup 
\begin{equation}
J_{F_A} \colonequals \langle \rho_0 \cdots \rho_n\rangle
\end{equation}
obtained as the cyclic subgroup of $\Aut(F_A)$ generated by the product of the generators $\rho_j$.  Then $J_{F_A}$ is the subgroup of $\Aut(F_A)$ that acts trivially on $X_A$.  

We now describe Berglund--H\"ubsch--Krawitz mirrors explicitly. Consider a group $G$ such that $J_{F_A} \subseteq G \subseteq \SL(F_A)$. Then we have a Calabi--Yau orbifold $Z_{A,G} \colonequals X_A / (G / J_{F_A})$. The mirror is given by looking at the polynomial $F_{A^T}$ obtained from the transposed matrix $A^T$ and the hypersurface $X_{A^T} \subset W\PP^n(q_0, \ldots, q_n)$, where $q_i$ are the dual weights.

As above, $\Aut(F_{A^T})$ is generated by the elements 
\[ \rho_j^T \colonequals (\exp(2\pi i b_{j0}), \ldots, \exp(2\pi i b_{jn})). \] 
We define the dual group to $G$ to be 
$$
G^T \colonequals \left\{ \prod_{j=0}^n (\rho_j^T)^{s_j} : \prod_{j=0}^n x^{s_j} \text{ is $G$-invariant} \right\} \subseteq \Aut(F_{A^T}). 
$$
Since $J_{F_A} \subseteq G \subseteq \SL(F_A)$, we have $J_{F_{A^T}} \subseteq G^T \subseteq \SL(F_{A^T})$ \cite[Proposition 3, Remark 3.2]{ABS14}.  Moreover, $J_{F_{A^T}}$ is generated by the element 
\begin{equation} \label{eqn:Jgen}
J^T \colonequals (\exp(2\pi i q_0/\dt), \ldots, \exp(2\pi i q_n/\dt)).
\end{equation}

Thus, we obtain a Calabi--Yau orbifold $Z_{A^T, G^T} \colonequals X_{A^T} / (G^T / J_{F_{A^T}})$. Berglund--H\"ubsch--Krawitz duality states that $Z_{A,G}$ and $Z_{A^T, G^T}$ are mirrors.

\begin{prop} \label{prop:dimChn1}
Let $X_{A,\psi}$ be an invertible \emph{pencil} of Calabi--Yau $(n-1)$-folds determined by the integer matrix $A$.  Then for all $\psi$ such that $X_{A,\psi}$, we have
\begin{equation} \label{eqn:dimCHn}
\dim_\C H_{\textup{prim}}^{n-1}(X_{A,\psi}, \C)^{\SL(F_A)}\geq\dt - \#I.
\end{equation}
\end{prop}

\begin{proof}
We have $\dim_\C H_{\textup{prim}}^{n-1}(X_{A,\psi}, \C)^{\SL(F_A)}\geq \dt - \#I$ since the Picard--Fuchs equation is $\SL(F_A)$-invariant.
\end{proof}


In certain cases, we have equality. We can compute $\dim_\C H_{\textup{prim}}^{n-1}(X_{A}, \C)^{\SL(F_A)}$ in the following way. Let 
\[ \mathscr{Q}_{F_A} \colonequals \frac{\C[x_0,\ldots,x_n]}{\left\langle \partial F_A/\partial x_0, \ldots, \partial F_A/\partial x_n \right\rangle} \]
be the \defi{Milnor ring} of $F_A$, i.e., the quotient of $\C[x_0,\ldots,x_n]$ by the Jacobian ideal.  A consequence of the Griffiths--Steenbrink formula \cite[Theorem 4.3.2]{Dol82} is that if $J_{F_A} \subseteq G$, then the $G$-invariant subspace of the Milnor ring viewed as a $\C$-module, $(\mathscr{Q}_{F_A})^G$, corresponds to the cohomology $H_{\textup{prim}}^{n-1}(X_A, \C)^{G}$.  

\begin{exm}\label{exm:Fermat}
Let $F_A = \sum_{i=0}^n x_i^{n+1}$ be the defining polynomial for the Fermat hypersurface $X_A \subseteq \PP^{n}$. Here, $\SL(F_A) = (\Z/ (n+1)\Z)^{n}$ where an element $(\xi_1,\ldots,\xi_n) \in \SL(F_A)$ acts by 
\[
(\xi_1, \ldots, \xi_{n}) \cdot (x_0, \ldots, x_n) \longmapsto (\xi_1 \cdots \xi_n x_0, \xi_1^{-1} x_1,\ldots,\xi_{n-1}^{-1} x_{n-1}, \xi_n^{-1} x_n).
\]
Note that in order for $\prod_i x_i^{a_i} \in \mathscr{Q}_{F_A}$ to be $\SL(F_A)$-invariant, it must satisfy the equalities $a_0 - a_i \equiv 0 \psmod{n+1}$ for all $i$. Thus the only $\SL(F_A)$-invariant elements of the Milnor ring are the $n$ elements $(x_0\cdots x_n)^a$ for $ 0\leq a < n$. Note that $n = \dt -\#I$, or the order of the Picard--Fuchs equation for this example, so equality holds in \eqref{eqn:dimCHn}.
\end{exm}

\subsection{Frobenius structure for the subspace associated to the holomorphic form} \label{sec:frobhol}

In this section, we will study a subspace $W_{\psi} \subset  H^{n+2}(\Omega_{\diamond,\psi}^{\bullet})$ generated by the connection acting on the holomorphic form. The dimension of this subspace is equal to the order of the Picard--Fuchs equation. It will in the end correspond to a factor 
\[
R_{\psi_0}(qT) \colonequals \det(1-\Frob T \mid W_{\psi})|_{\psi=\widehat{\psi_0}}
\]
of the zeta function for $X_{A, \psi_0}$, where $X_{A,\psi_0}$ is a nondegenerate and smooth member of the pencil. 
We prove in this section that there is a Frobenius structure on $W_\psi$ by examination of the unit root.

Let $K = k(\psi)$ where $k$ is an algebraically closed field of characteristic 0. Let $D \colonequals \partial/\partial \psi$ be the standard derivation on $K$. The space $H^{n+2}(\Omega_{\diamond,\psi}^{\bullet})$ is a differential module that is finite-dimensional over $K$ with connection $\nabla:=\nabla(D)$. Consider the submodule $W_\psi \subseteq H^{n+2}(\Omega_{\diamond,\psi}^{\bullet})$ obtained by repeatedly applying $\nabla$ to the holomorphic form defined by the monomial $\xi_0 = wx_0 x_1 \cdots x_n$.  Then $W_\psi$ is a $\nabla$-stable subspace with cyclic basis $\xi_0, \ldots, \xi_{N-1}, \ \text{with} \ \xi_j = \nabla^j(\xi_0)$. so that
$$
\nabla \begin{pmatrix} \xi_0 \\ \vdots \\ \xi_{N-1} \end{pmatrix} = G^T \begin{pmatrix} \xi_0 \\ \vdots \\ \xi_{N-1} \end{pmatrix}
$$
where 
$$
G = \begin{pmatrix} 0  & \hdots & 0 & g_{N-1} \\ 1 & \hdots & 0  & g_{N-2} \\ \vdots & \ddots & \vdots & \vdots \\ 0 &  \hdots & 1 & g_0 \end{pmatrix},
$$
with $g_i \in K$.

In the theory of Dwork, the Frobenius structure on the differential equation arises in the dual theory. Say $\mathcal{K}_{\psi}$ is dual to  $H^{n+2}(\Omega_{\diamond,\psi}^{\bullet}$ and let $W_\psi^*$ be dual to $W_{\psi}$. It is a differential module over $K$ with connection $\nabla^*$ satisfying the pairing 
$$
D(\xi, \xi^*) = (\nabla(D)\xi, \xi^*) + (\xi, \nabla^*(D) \xi^*)
$$
for $\xi \in W_\psi, \xi \in W_\psi^*$. Via the dual basis, we obtain the  connection acting on the dual basis $\nabla^*$ on $W_\psi^*$:
$$
\nabla^* \begin{pmatrix} \xi_0^* \\ \vdots \\ \xi_{N-1}^* \end{pmatrix} = -G\begin{pmatrix} \xi_0^* \\ \vdots \\ \xi_{N-1}^* \end{pmatrix}.
$$
A horizontal section $\zeta^* = \sum_{i=0}^{N-1} C_i(\psi) \xi^*_i$  under $\nabla^*$ has coefficients $\{C_0(\psi), \ldots, C_{N-1}(\psi)\}$ which satisfy the differential equation
$$
D(C_0(\psi), \ldots, C_{N-1}(\psi)) = (C_0(\psi), \ldots, C_{N-1}(\psi))G.
$$
Note that in our case working with a cyclic basis $C_0(\psi)$ is a solution of the scalar differential equation $Ly= 0 \ \text{where} \ L= D^N - \sum_{i=0}^{N-1} g_i(\psi) D^i$.

By Proposition~\ref{prop:PFirred}, the operator $L$ is irreducible.

\begin{prop}\label{irredDiffModule}
$W_\psi^*$ contains no nonzero, proper $\nabla^*$-stable differential submodule.
\end{prop}
\begin{proof}
This proposition is proven by Sabbah \cite[Theorem 2.4]{Sabbah}; for completeness, we provide an argument here. Suppose $M_0^*$ were such a nonzero, proper $\nabla^*$-stable differential submodule of dimension $0 < r <N$. We will show that if such a proper submodule existed, then the Picard--Fuchs operator $L(D)$ has a proper factorization in the noncommutative polynomial ring $K[D]$ and the Picard-Fuchs equation would necessarily be reducible, contradicting Proposition~\ref{prop:PFirred}.

Without loss of generality we may assume $M_0^*$ has a cyclic basis $\{\gamma_0^*, \nabla \gamma_0^*, \ldots, (\nabla^*)^{r-1}\gamma_0^*\}$.  Choose elements $\delta_r^*, \ldots, \delta_{N-1}^* \in W_{\psi}^*$ so that the set $\{\gamma_0, \nabla \gamma_0, \ldots, \nabla^{r-1}\gamma_0, \delta_r, \ldots, \delta_{N-1}\}$ is a basis for $W_{\psi}$.  Then we can write the connection matrix for $W_{\psi}^*$ in the form:

$$
\nabla^*   \left(\begin{array}{@{}c@{}}
    \gamma_0^* \\ \vdots \\ (\nabla^*)^{r-1}\gamma_0^* \\\hline \delta_j^*  
  \end{array}\right)  = -H
  \left(\begin{array}{@{}c@{}}
    \gamma_0^* \\ \vdots \\ (\nabla^*)^{r-1}\gamma_0^* \\\hline \delta_j^*  
  \end{array}\right), \text{ where }H:= \left(\begin{array}{@{}cccc|c@{}}
    0 & \cdots & 0 & h_{r-1} & *\\
    1 & \cdots & 0 & h_{r-2} &* \\
    \vdots & \ddots & \vdots & \vdots & * \\
    0 & \cdots & 1 & h_0 & * \\\hline
    0 & 0 & 0 & 0 & * 
  \end{array}\right),
$$
where $\{h_i\}_{i=0}^{r-1} \subset K$.
We consider a horizontal section 
$$
\sum_{i=0}^{r-1} B_i (\nabla^*)^i \gamma_0^* + \sum_{i=r}^{N-1} B_i \delta_i^*,
$$
for some $B_i(\psi) \in K$ so that $D(B_0, \ldots, B_{N-1}) = (B_0, \ldots, B_{N-1}) H$, and 
$$
D(B_0, \ldots, B_{r-1}) = (B_0, \ldots, B_{r-1}) \left(\begin{array}{@{}cccc@{}}
    0 & \cdots & 0 & h_{r-1} \\
    1 & \cdots & 0 & h_{r-2} \\
    \vdots & \ddots & \vdots & \vdots  \\
    0 & \cdots & 1 & h_0  \\
  \end{array}\right).
$$
So the entries $\{B_0,...B_{r-1} \}$ are dependent over K. We now can rewrite this horizontal section in terms of our original dual basis
$$
\sum_{i=0}^{r-1} B_i (\nabla^*)^i \gamma_0^* + \sum_{i=r}^{N-1} B_i \delta_i^*= \sum_{i=0}^{N-1} A_i (\nabla^*)^i \xi_i^*,
$$
for some $A_i \in K$. Note that $A_0$ must be a solution of the Picard--Fuchs differential equation.  There exists some nonsingular matrix $\mathcal{A}$ over $K$ so that 
$$
(\gamma_0^*, \ldots (\nabla^*)^{r-1}\gamma_0^*, \delta_{r}^*, \ldots, \delta_{N-1}^*)^T = \mathcal{A} (\xi_0^*, \ldots, \xi_{N-1}^*)^T.
$$
Using this change of basis, we can see that 
$$
(B_0, \ldots, B_{N-1}) \mathcal{A} = (A_0, \ldots, A_{N-1})
$$
where $A_i = D^i A_0$, since $\sum_{i=0} A_i (\nabla^*)^i \xi_i^*$ is a horizontal section. This gives a non-trivial homogeneous relation among $A_0, \ldots, D^{N-1}A_0$; thus, $A_0$ satisfies a lower order differential equation defined over $K$. Using the usual argument via the division algorithm in the noncommutative ring $K[D]$ we conclude that the Picard-Fuchs operator  has a non-trivial right factor in $K[D]$ which contradicts the irreducibility of the Picard--Fuchs equation.
\end{proof} 

\begin{lem}\label{lem:ExistenceFrobenius}
Let $\psi \in \PP^1$ be such that $X_{A,\psi}$ is nondegenerate and smooth. Then there exists a strong Frobenius structure on $W_\psi^*$.
\end{lem}
\begin{proof}
We recall section \ref{sec:unitroot}.  Suppose that $X_\psi$ is ordinary, a condition that holds for all but finitely many $\psi \in \overline{\F}_p$.  Then there is a unique unit root of the characteristic polynomial of Frobenius acting on $H^{n+2}(\Omega_{\diamond,\psi}^{\bullet})$, and this yields a unique unit root eigenvector $\eta_0$ up to scaling.  The same holds for the dual space $\mathcal{K}_{\psi}$ of $H^{n+2}(\Omega_{\diamond,\psi}^{\bullet})$ with unique unit eigenvector $\eta_0^*$. We  claim that $\eta_0^* \in W_\psi^*$.  Assume for the purposes of contradiction that $\eta_0^* \not\in W_\psi^*$.  Then  we may take as a basis for  $\mathcal{K}_{\psi}$  a set containing $\eta_0^*$ and the cyclic basis $\{(\nabla^*)^i\}_{i=0}^{N-1} \omega_0^*$ for $W_{\psi}^*$ as in \eqref{eqn:defofomega0}. Let $A^*$ be the matrix of $q^{-1}$-Frobenius in this basis.  Since $\eta_0^*$ is a unit eigenvector, the diagonal coefficient of $A^*$ corresponding to $\eta_0^*$ is nonzero modulo $p$ (and the other coefficients of this column are zero).  But by \eqref{eqn:anunu}, the diagonal coefficient of $A^*$ for $\omega_0^*$ is nonzero modulo $p$ because $X_\psi$ is ordinary.  Therefore $A^*$ has rank at least $2$ modulo $p$, and this contradicts Lemma \ref{lem:qinFrob}.  

  So now let $\eta_0^* \in W_\psi^*$ be the unit root eigenvector, unique up to scaling and defined on the ordinary locus $U\subseteq \mathbb{P}^1$ where $U$ is the complement of the union of $\{0,1,\infty\}$ and the supersingular locus for the given pencil $X_{A,\psi}$. Then writing $\operatorname{Frob}$ for $q^{-1}$-Frobenius
$$
\operatorname{Frob}  \eta_0^* = u \eta_0^*,
$$
where $u \in K$ is a unit on the locus $U$. Frobenius commutes with the connection $\nabla^*$, so 
$$
\operatorname{Frob} (\nabla^* \eta_0^*) = \nabla^* \operatorname{Frob} \eta_0^* = u (\nabla^* \eta_0^*) + D(u)\eta_0^*,
$$
which implies that  Frobenius is stable on the submodule that is generated by the cyclic basis given by $\{(\nabla^*)^i \eta_0^* \ | \ i \in \mathbb{Z}_{\geq0}\}$, but this is $W_\psi^*$ by Proposition~\ref{irredDiffModule}. Hence, for each choice of pencil indexed by $\diamond$ the Picard-Fuchs equation has a strong Frobenius structure in the sense of Dwork \cite{uniqueness}.
\end{proof}


\subsection{Proof of main result}

In this section, we prove our main result.  We will make use of the following lemma.

\begin{lem} \label{lem:XFqproj}
Let $X$ be a projective variety over $\F_q$ and let $G$ be a finite group of automorphisms of $\overline{X} = X \times_{\F_q} \overline{\F_q}$ stable under $\Gal(\overline{\F_q}/\F_q)$.  Then the following statements hold.
\begin{enumalph}
\item The quotient $\overline{X}/G$ exists as a projective variety over $\F_q$.  
\item Let $\ell \neq p$ be prime and suppose $\gcd(\#G,\ell) = 1$.  Then for all $i$, the natural map
\[ H_{\textup{\'et}}^i(\overline{X}/G,\Q_\ell) \xrightarrow{\sim} H_{\textup{\'et}}^i(\overline{X},\Q_\ell)^G \]
is an isomorphism.
\end{enumalph}
\end{lem}

\begin{proof}
See Harder--Narasimhan \cite[Proposition 3.2.1]{HarderNara} (with some extra descent).
\end{proof}

Our main result (slightly stronger than Theorem \ref{thm:sameweights}) is as follows.


\begin{thm}  \label{thm:commonkmf}
Let $X_{A,\psi}$ and $X_{B,\psi}$ be invertible pencils of Calabi--Yau $(n-1)$-folds in $\mathbb{P}^n$.  Suppose $A$ and $B$ have the same dual weights $(q_i)_i$.  Then for each $\psi \in \F_q$ such that $\gcd(q,(n+1)d^T)=1$ and the fibers $X_{A,\psi}$ and $X_{B,\psi}$ are nondegenerate and smooth, 
there exists a polynomial  $R_\psi(T) \in \Q[T]$ with 
\[ D(q_0,\dots,q_n) \leq \deg R_{\psi}(T) \leq \dim_\mathbb{C} H_{\textup{prim}}^{n-1}(X_{A,\psi}, \mathbb{C})^{\operatorname{SL}(F_A)}. \]
such that $R_\psi(T)$ divides  $P_{X_{A, \psi}}(T)$ and $P_{X_{B,\psi}}(T)$.
\end{thm}

\begin{proof}
Let $F_{\diamond,\psi}(x)$ be invertible pencils, corresponding to matrices $\diamond=A,B$ with the same weights.  Then by Theorem~\ref{thm: Gahrs}, the Picard--Fuchs equations are of order $D(q_1,\dots,q_n)$ are the same. Suppose that the two pencils have a common smooth fiber $\psi \in \F_q$.  

We follow the construction of cohomology in Adolphson--Sperber \cite{AS08}, with a few minor modifications.  We assume their base field $\Lambda_1$ is enlarged to treat $\psi$ as a variable over $\Q_p(\zeta_p)$ with (unit) $p$-adic absolute value, so that $\Lambda_1$ has $\partial/\partial \psi$ as a nontrivial derivation.   Then the construction of the complex $\Omega^{\bullet}_\psi$ is unchanged as are the cohomology spaces $H^i(\Omega^{\bullet}_\psi)$.  Then \cite[Theorem 6.4, Corollary 6.5]{AS08}
\[ P_{\diamond,\psi_0}(qT) \colonequals \det(1-\Frob T \mid H^{n+2}(\Omega_{\diamond,\psi}^{\bullet}))|_{\psi=\widehat{\psi_0}}, \]
where $\widehat{\psi_0}$ is the Teichm\"uller lift of $\psi_0$.  

The connection
\[ \epsilon\left(\psi \frac{\partial}{\partial\psi}\right) = \psi \frac{\partial}{\partial \psi} - \dt\psi x_0 x_1 \cdots x_n \]
acts on $H^{n+2}(\Omega_{\diamond,\psi}^{\bullet})$.  By work of Katz \cite{Kat68}, the associated differential equation is the Picard--Fuchs equation.  

For each invertible pencil determined by a choice of $\diamond$, as in section \ref{sec:frobhol}, we have a subspace $W_\psi$ obtained by repeatedly applying the connection to the monomial $wx_0 x_1 \cdots x_n$ corresponding to the holomorphic form.  By Lemma \ref{lem:ExistenceFrobenius}, we obtain a strong Frobenius structure on this differential module.  By construction, the associated differential equation is the hypergeometric Picard--Fuchs equation, and this equation is \emph{independent} of $\diamond$ by Theorem \ref{thm: Gahrs}.  By Proposition~\ref{prop:PFirred}, this differential equation is irreducible. 
Under the hypothesis that $p \nmid (n+1)\dt$, there is a $p$-integral solution to this differential equation.  Then by a result of Dwork \cite[Lemma, p.\ 89--90]{uniqueness}, the respective Frobenius matrices $\Phi_{\diamond,\psi_0}$ acting on $W$ differ by $p$-adic constant.  As in the proof of Proposition \ref{prop:uniqueoffrob}, the same unique unit root at a smooth specialization implies that this constant is $1$. 

At the same time, the subspace $\Sigma_{\diamond,\psi} \colonequals H^{n+2}(\Omega_{\diamond,\psi}^{\bullet})^{\SL(F_A)}$ invariant under $\SL(F_A)$ is stable under the connection and has an action of Frobenius.  The group $\SL(F_A)$ preserves the holomorphic form, so $W_{\psi} \subseteq \Sigma_{\diamond,\psi}$.

Let
\begin{equation}
\begin{aligned}
R_{\psi_0}(qT) &\colonequals \det(1-\Frob T \mid W_{\psi}))|_{\psi=\widehat{\psi_0}} \\
S_{\diamond,\psi_0}(qT)  &\colonequals \det(1-\Frob T \mid \Sigma_{\diamond,\psi}))|_{\psi=\widehat{\psi_0}}.
\end{aligned}
\end{equation}
We have shown that 
\[ R_{\psi_0}(T) \mid S_{\diamond,\psi_0}(T) \mid P_{\diamond,\psi_0}(T) \]
with $R_{\psi_0}(T)$ independent of $\diamond$.  Since $P_{\diamond,\psi_0}(T) \in \Q[T]$, as it is a factor of the zeta function, we know immediately that $R_{\psi_0}(T) \in K[T]$ for $K$ a number field, which we may assume is Galois over $\Q$ by enlarging.

Next, we apply Lemma \ref{lem:XFqproj}: the characteristic polynomial of Frobenius via the Galois action on $H^{n-1}_{\textup{\'et}}(\overline{X}_{A,\psi_0},\Q_\ell)^{\SL(F_A)}$ is equal to $S_{\diamond,\psi_0}(qT)$.  Therefore $S_{\diamond,\psi_0}(qT) \in \Q_\ell[T]$ for all but finitely many $\ell$, and so is independent of $\ell$ and it also belongs to $\Q[T]$.  Now let 
\[ R'_{\psi_0}(T) \colonequals \lcm_{\sigma \in \Gal(K/\Q)} \sigma(R_{\psi_0})(T) \] 
be the least common multiple of the polynomials obtained by applying $\Gal(K/\Q)$ to the coefficients of $R_{\psi_0}$. Then $R'_{\psi_0}(T)$ is still independent of $\diamond$, by Galois theory $R'_{\psi_0}(T) \in \Q[T]$, and $R'_{\psi_0}(T) \mid S_{\diamond,\psi_0}(T) \mid P_{\diamond,\psi_0}(T)$ is a factor of the zeta function and
\[ \dt - \#I = \deg R_{\psi_0}(T) \leq \deg R'_{\psi_0}(T) \leq \deg S_{\diamond,\psi_0}(T) = H_{\textup{prim}}^{n-1}(X_{A,\psi}, \C)^{\SL(F_A)} \]
as desired.
\end{proof}

\begin{cor} \label{cor:dworksolved}
With hypotheses as in Theorem \textup{\ref{thm:commonkmf}}, suppose that the common dual weights are $(q_0,\ldots,q_n)=(1,\ldots,1)$.  Then $\deg R_{\psi}(T) = n$.
\end{cor}

\begin{proof}
First note that $D(1, \ldots, 1) = n$. By Example~\ref{exm:Fermat}, we know that for the Dwork pencil, we have the equality $D(q_0, \ldots, q_n) = \dim H_{\textup{prim}}^{n-1}(X_{A})^{\SL(F_A)}$. By applying Theorem~\ref{thm:commonkmf} to first obtain the common factor $R_{\psi}(T)$ and then applying Theorem \ref{thm:commonkmf}, we then have that $R_{\psi}(T) \in \Q[T]$ and is of degree  $D(1, \ldots, 1) = n$.
\end{proof}

In particular, by a straightforward calculation, if the invertible pencil consists of only Fermats and loops (no chains), then the dual weights are $(1,\ldots,1)$ and Corollary \ref{cor:dworksolved} applies.

\begin{rmk} \label{rmk:hypergeom}
It is also possible to argue for a descent to $\Q[T]$ of a common factor of degree $\dt-\#I$ purely in terms of hypergeometric motives---without involving the group action---as follows.  First, we need to ensure that the trace of Frobenius on the subspace of $p$-adic cohomology cut out by the hypergeometric Picard--Fuchs equation is given by an appropriately normalized finite field hypergeometric sum: this is implicit in work of Katz \cite[\S 8.2]{Katz:ESDE} and should be implied by rigidity \cite[\S 8.10]{Katz:ESDE}, but we could not find a theorem that would allow us to conclude this purely in terms of the differential equation.

In such a situation, by an elementary observation (found in Beukers--Cohen--Mellit \cite[p.\ 3]{BCM}), these hypergeometric sums are defined over $\Q$ if and only if the polynomials 
\[ g_{\pmb{\alpha}} \colonequals  \prod_{\alpha_i \in  \pmb{\alpha}\smallsetminus I } (x- e^{2\pi\sqrt{-1} \alpha_i}),\qquad g_{\pmb{\beta}} \colonequals  \prod_{\beta_{ij} \in \pmb{\beta} \smallsetminus \pmb{I}} (x- e^{2\pi\sqrt{-1} \beta_{ij}}) \]
belong to $\Z[T]$.  This statement can be shown directly.  

We show this invariance first for the polynomial $g_{\pmb{\alpha}}$.  Let $r_i = \gcd(q_i, \dt)$ for $i=0,\dots,n$. Consider the set 
\[ K=\{k : \text{$k>1$ and $k \mid \dt$ and $k \nmid r_i$ for all $i=0,\dots,n$}\}. \]
Then $\pmb{\alpha} \smallsetminus I = \{ j/k  : k \in K, \gcd(j, k) = 1\}$ so 
\begin{equation} 
g_{\pmb{\alpha}} = \prod_{k \in K} \Phi_{k}(x) 
\end{equation}
where $\Phi_k(x) \in \Z[T]$ is the $k$th cyclotomic polynomial, as desired.

A similar argument works for $g_{\pmb{\beta}}$.  Let $r_{ij} = \gcd(r_i, r_j)$ for $i, j = 0, \ldots, n$.  For each $i=0,\dots,n$, let
$$
K_i = \{ k_i : \text{$k_i \mid q_i$ and $k_i \nmid r_i$}\} \pmb{\cup} \{k_{ij} :  \text{$k_{ij} \mid r_i$ and $k_{ij} \mid r_{ij}$ for some $j<i$}\}.
$$
Then 
\[ \pmb{\beta}\smallsetminus I = \bigcup_{i=1}^n \{0\} \pmb{\cup} \bigcup_{i=0}^n \left\{ j/k_i : k_i \in K_i\text{ and }\gcd(j, k_i) = 1\right\}. \]
Hence 
\begin{equation} 
g_{\pmb{\beta}} = (x-1)^{n}\prod_{i = 0}^n \prod_{k \in K_i} \Phi_{k}(x) \in \Z[T]. 
\end{equation}
\end{rmk}

\begin{rmk}
There is yet a third way to observe a common factor purely in terms of group invariance using a common cover by a Fermat pencil (of larger degree): see recent work of Kloosterman \cite{Kloosterman:preprint}.
\end{rmk}

\subsection{Unit roots and point counts}

If $X$ is a smooth Calabi--Yau variety, the polynomial $P_X(T)$ appearing in the zeta function of $X$ has at most one root that is a $p$-adic unit.  This root is called the \defi{unit root}.  We have already used the unit root implicitly to compare zeta functions.  We may also use the unit root directly to extract arithmetic information about an invertible pencil from $A^T$.  This yields a simple arithmetic relationship between different invertible pencils with the same dual weights.

\begin{prop}\label{prop:unitroot}
Let $F_A(x)$ and $F_B(x)$ be invertible polynomials in $n+1$ variables satisfying the Calabi--Yau condition.  Suppose $A^T$ and $B^T$ have the same weights.  Then for all $\psi \in \mathbb{F}_q$ and in all characteristics including when $p \mid \dt$, either the unit root of $X_{A,\psi}$ is the same as the unit root of $X_{B,\psi}$, or neither variety has a nontrivial unit root.  Thus, the supersingular locus is the same for both pencils.
\end{prop}

\begin{rmk}
In the case of non-smooth, non-supersingular fibers, Adolphson--Sperber \cite{AS16} describe what is meant here by the unit root and show that then the unit root is given by the same formula as in the smooth case.  Dwork noted the possibility of a meaningful unit root formula for varieties that are not smooth \cite{Dwo62}. 
\end{rmk}

\begin{proof} 
In the case where $p$ divides $\dt$ we replace $\dt\psi$ in the given families by $\psi$ in order to obtain a nontrivial pencil.  Adolphson--Sperber \cite{AS16} provide a formula for the unit root using $A$-hypergeometric functions. The lattice of relations used to compute the $A$-hypergeometric functions is determined by the dual weights, and the character vector is the same in both families. Thus, the unit root formula is the same in both cases.  More precisely, in the case of smooth fibers, the middle dimensional factor has a unique unit root which occurs in the common factor $R_\psi(T)$ described above. It is given by a $p$-adic analytic formula in terms of the series defined above. 
The Hasse invariant is determined by the reduction of the $A$-hypergeometric series solution mod $p$. This proves the identity of the supersingular locus in cases where the weights agree.  
\end{proof}

\begin{rmk}
In the case that $\psi \in \mathbb{F}_q^\times$ yields a smooth member of the pencil $X_{A,\psi}$, the result of Proposition~\ref{prop:unitroot} can also be obtained from Miyatani \cite[Theorem 2.9]{Miyatani}, where the unit root is nontrivial precisely when a formal power series defined using the hypergeometric parameters appearing in Equation~\ref{equ:genhypergeometric} is nonzero.  Miyatani also gives a formula for the unit root when it exists and $X_{A,\psi}$ is smooth, in terms of the same hypergeometric power series.  As we have already observed, the hypergeometric parameters depend only on the weights of $A^T$ or $B^T$.
\end{rmk}

Proposition~\ref{prop:unitroot} implies a relationship between point counts for alternate mirrors, reminiscent of Wan's strong arithmetic mirror symmetry \cite{fw, wan}.

\begin{cor}
Let $F_A(x)$ and $F_B(x)$ be invertible polynomials in $n+1$ variables satisfying the Calabi--Yau condition.  Suppose $A^T$ and $B^T$ have the same weights.  Then for any fixed $\psi \in \mathbb{F}_q$ and in all characteristics (including $p \mid \dt$) the $\mathbb{F}_q$-rational point counts for fibers $X_{A,\psi}$ and $X_{B,\psi}$ are congruent as follows: 
\[\#X_{A,\psi}(\F_q) \equiv \#X_{B,\psi}(\F_q) \pmod{q}.\]
\end{cor}

\begin{proof}
The formula is true vacuously when the fiber is supersingular (there is no unit root). Otherwise, the unit root controls the point count modulo $q$.  \end{proof}
 
The congruence result given here is weaker of course for smooth fibers than the result given earlier on common factors, Theorem~\ref{thm:commonkmf} above. It is possible that the common factor result for the piece of middle dimensional cohomology invariant under the respective group actions does extend meaningfully to fibers that are not smooth as well. Computations in \cite{kadir, kadir2, CORV2} show that a factor of the zeta function associated to the holomorphic form can be identified for singular fibers of the Dwork pencils of quartics and quintics, as well as for a certain family of octic Calabi--Yau threefolds in a weighted projective space.  We expect there will be a common factor (for families with the same dual weights) for singular fibers in the case of K3 surfaces, since the unit root in this case should govern the relevant factor (using the functional equation and the fact that the determinant of Frobenius is constant).  

\section{Quartic K3 Surfaces}

We now specialize to the case of $n=3$, i.e., K3 surfaces realized as a smooth quartic hypersurface in $\mathbb{P}^3$. 

\subsection{Pencils of K3 surfaces}

The invertible pencils in $\PP^3$ whose Berglund--H\"ubsch--Krawitz mirrors are hypersurfaces in finite quotients of $\PP^3$ are listed in the following table.  We list the group of \defi{symplectic} symmetries $SL(F_A)/J_{F_A}$, which act nontrivially on each projective hypersurface and fix its holomorphic form, in the third column.
\begin{equation} \label{table:5families}
\begin{tabular}{c|c|c}
 Family & Equation for $X_{A,\psi}$	&  Symmetries \\
\hline	\hline
\rule{0pt}{2.5ex}   $\Fsf_4$ & $x_0^4+x_1^4 + x_2^4 + x_3^4 - 4\psi x_0x_1x_2x_3$ 	& $(\Z/4\Z)^2$	 \\
 $\Fsf_2\Lsf_2$ & $x_0^4 + x_1^4 + x_2^3x_3 + x_3^3x_2 - 4\psi x_0x_1x_2x_3$	& $\Z/8\Z$ \\
 $\Fsf_1\Lsf_3$ & $x_0^4 + x_1^3x_2 + x_2^3x_3 + x_3^3x_1 - 4\psi x_0x_1x_2x_3$ & $\Z / 7\Z$  \\
 $\Lsf_2\Lsf_2$ & $x_0^3x_1 + x_1^3x_0 + x_2^3x_3 + x_3^3x_2 - 4\psi x_0x_1x_2x_3$ & $\Z/4\Z \times \Z/2\Z$ \\
 $\Lsf_4$ & $x_0^3x_1 + x_1^3x_2 + x_2^3x_3 + x_3^3 x_0 - 4\psi x_0x_1x_2x_3$ & $\Z / 5\Z$  \\
\end{tabular}
\end{equation}   

Recalling Example~\ref{example: K3QuarticsPF}, we observe that each of these five pencils has the same degree three Picard--Fuchs equation for the holomorphic form, and that after a change of variables, this equation is the differential equation satisfied by the classical hypergeometric function
\begin{equation} \label{eqn:K3hypnm}
\phantom{i} {}_3F_{2}\left(\begin{array}{c}\frac{1}{4},\frac{1}{2},\frac{3}{4} \\1, 1 \end{array}; \psi^{-4}\right).
\end{equation}
  
The main result of this section is the following theorem.

\begin{thm} \label{thm:K3thm}
Let $\diamond \in \calF=\{\Fsf_4,\Fsf_2\Lsf_2,\Fsf_1\Lsf_3,\Lsf_2\Lsf_2,\Lsf_4\}$ signify one of the five \textup{K3}\/ families in Table \/\textup{\ref{table:5families}}.  Let $q=p^r$ be a prime power with $p \neq 2,5,7$ and let $\psi \in \F_q$ be such that $\psi^4 \neq 1$.  Then $X_{\diamond,\psi}$ is a smooth, nondegenerate fiber of the family $\diamond$.

Let $P_{\diamond,\psi,q}(T) \in 1+T\Z[T]$ be the nontrivial factor of $Z(X_{\diamond,\psi}/\F_q,T)$ of degree $21$.  Then the following statements hold.
\begin{enumalph}
\item We have a factorization
\[ P_{\diamond,\psi,q}(T) = Q_{\diamond,\psi,q}(T)R_{\psi,q}(T) \]
in $\Z[T]$ with $\deg Q_{\diamond,\psi,q}=18$ and $\deg R_{\psi,q}=3$.  
\item The reciprocal roots of $Q_{\diamond,\psi,q}(T)$ are of the form $q$ times a root of $1$.  
\item The polynomial $R_{\psi,q}(T)$ is independent of $\diamond \in \calF$.
\end{enumalph} 
\end{thm}

\begin{rmk}
In future work \cite{paperB}, we study these families in more detail: we describe a further factorization of $Q_{\diamond,\psi,q}(T)$ related to the action of each group, and we identify each of these additional factors as hypergeometric.
\end{rmk}

The polynomials $P_{\diamond,\psi,q}(T)$ have degree $21$ and all of their reciprocal roots $\alpha$ satisfy $|\alpha|= q$, by the Weil conjectures.  By a direct calculation in the computer algebra system \textsc{Magma} \cite{Magma}, when $p \neq 2,5,7$ and $\psi^4 \neq 1$, the fiber $X_{\diamond,\psi}$ is smooth and nondegenerate.  Parts~(a) and~(c) of Theorem~\ref{thm:K3thm} now follow from Theorem~\ref{thm:sameweights} and the Picard--Fuchs differential equation computed in Example~\ref{example: K3QuarticsPF}.

We now prove Theorem~\ref{thm:K3thm}(b).
For all $\diamond \in \calF$, the trace formula \eqref{eqn:PqT} asserts that
\[ P_{\diamond,\psi,q}(T) = \det(1-\Frob T \mid H^4(\Omega^{\bullet}_{\diamond})). \]
We now analyze the unit root.  In section \ref{sec:unitroot}, we saw that there is at most one unit root of $P_{\diamond,\psi,q}(T)$.  If there is no unit root, then the K3 surface $X_{\diamond,\psi}$ is supersingular over $\F_q$, and Theorem~\ref{thm:K3thm}(b) follows by the Tate conjecture for K3 surfaces.  Thus, we need only analyze the case where there is a unit root.

\begin{prop}  \label{prop:Pdiamondq}
Suppose $P_{\diamond,\psi,q}(T)$ has a unit root $u(\psi)$.  Then the reciprocal zeros $\beta=\beta_{\diamond}$ of $P_{\diamond,\psi,q}(T)$ other than $u(\psi)$ and the root $ q^2/u(\psi)$ all have the form $\beta=q\zeta$ where $\zeta$ is a root of unity. \end{prop}

\begin{proof}
We know that $\beta$ is an algebraic integer which by Deligne's proof of the Riemann hypothesis has the form $\beta = q \zeta$ with $\zeta$ an algebraic number with complex absolute value $|\zeta|_\infty=1$.  By the functional equation $\beta\beta'=q^2$, so that for any prime $\ell \neq p$, we have that $\beta$ (and $\zeta$) are $\ell$-adic units.  Since we are considering now only ordinary fibers $\psi$, the first slope of Newton agrees with the first slope of Hodge. It then follows for every $\beta$ a reciprocal zero of $P_{\diamond}(t)$ other than the unit root $u(\psi)$, we have $\ord_q(\beta) \geq 1$. As a consequence, $\zeta$ is a $p$-adic integer.  This proves $\zeta$ is an algebraic integer.  From the product formula $|\zeta|_p=1$. We have shown that $|\zeta|_v=1$ for all places $v$ of $\Q$.  By Dirichlet's theorem, this implies $\zeta$ is a root of unity.
\end{proof}

Before concluding this section, we consider the remaining invertible quartic pencils in $\PP^3$.  We may use methods similar to the analysis of Theorem~\ref{thm:K3thm} to relate two pencils of K3 surfaces whose equations incorporate chains.
\begin{equation} \label{table:chainfamilies}
\begin{tabular}{c|c|c}
Family	& Equation for $X_{A,\psi}$	& Symmetries \\
\hline \hline
\rule{0pt}{2.5ex} 
$\Csf_2\Fsf_2$	& $ x_0^3x_1+x_1^4 + x_2^4 + x_3^4 - 12\psi x_0x_1x_2x_3$ 	&$ \Z/4\Z$  \\

$\Csf_2 \Lsf_2$	& $x_0^3x_1 + x_1^4 + x_2^3x_3 + x_3^3 x_2 - 12\psi x_0x_1x_2x_3$ &  $\Z/2\Z$	
\end{tabular}
\end{equation}

Let $\clubsuit \in \mathcal{G}=\{\Csf_2\Fsf_2, \Csf_2 \Lsf_2\}$ signify one of the two \textup{K3}\/ families in Table \/\textup{\ref{table:chainfamilies}}.  The dual weights for these families are $(4,2,3,3)$.  Let $X_{\clubsuit,\psi}$ be a smooth member of $\clubsuit$, and assume $\gcd(q,6)=1$.  Let $P_{\clubsuit,\psi}(T) \in 1+T\Z[T]$ be the nontrivial factor of $Z(X_{\clubsuit,\psi},T)$ of degree $21$ as in \eqref{eqn:PXT}.  Then by Theorem \ref{thm:commonkmf} we have a factorization
\begin{equation}\label{equ:clubs}
P_{\clubsuit,\psi}(T) = Q_{\clubsuit,\psi}(T)R_{\psi}(T) 
\end{equation} 
in $\Z[T]$ with $6 \leq \deg R_{\psi} \leq 7$ and $R_{\psi}(T)$ is independent of $\clubsuit \in \mathcal{G}$.  However, we pin this down in the next subsection, and show in fact that $\deg R_{\psi}=6$ (as expected), with $\deg Q_{\clubsuit,\psi}=15$.  The reciprocal roots of $Q_{\clubsuit,\psi}(T)$ are of the form $q$ times a root of $1$ from a similar argument as in Proposition \ref{prop:Pdiamondq}.  

Together, Theorem~\ref{thm:K3thm} and Equation~\ref{equ:clubs} give a complete description of the implications of Theorem~\ref{thm:sameweights} for invertible pencils of K3 hypersurfaces in $\PP^3$; the remaining three pencils, classified for example by Doran--Garavuso  \cite{DG}, are each described by matrices with distinct sets of dual weights.

\subsection{Discussion and applications}

By Tate's conjecture, a theorem due to work of Charles \cite{Charles}, Madapusi Pera \cite{Pera}, and Kim--Madapusi Pera \cite{KimPera}, the N\'eron--Severi rank of a K3 surface $X$ over $\F_q$ is equal to one plus the multiplicity of $q$ as a reciprocal root of $P(T)$ \cite[Corollary 2.3]{vL}, and this rank is even.  (The extra ``one'' corresponds to the hyperplane section, already factored in.)  Thus Theorem~\ref{thm:K3thm}(b) implies that each $X_{\diamond,\psi}$ has N\'eron--Severi rank over the algebraic closure $\overline{\F}_q$ at least $18+1=19$, so at least $20$ because it is even.  Similarly, each $X_{\clubsuit,\psi}$ has N\'eron--Severi rank over $\overline{\F}_q$ at least $14+1=15$, thus $16$ because it is even.  

By comparison, in characteristic $0$ we can inspect the N\'eron--Severi ranks as follows.  Theorem~\ref{thm:K3thm} implies that the subspace in cohomology cut out by the Picard--Fuchs equation is contained in the $\SL(F_A)$-invariant subspace and it contains $H^{2,0}$.  Consequently, as observed by Kloosterman \cite{Kloosterman:preprint}, this implies that the $\SL(F_A)$-invariant subspace in $H^2_{\textup{\'et}}(X_{A,\psi})$ contains the transcendental subspace: indeed, one definition of the transcendental lattice of a K3 surface is as the minimal primitive sub-$\Q$-Hodge structure containing $H^{2,0}$ \cite[Definition 3.2.5]{Huybrechts}.  

For the five pencils in Table \ref{table:5families} with dual weights $(1,1,1,1)$, we conclude that the generic N\'eron--Severi rank is at least $22-3=19$; but it cannot be $20$, because then the family would be isotrivial, so it is equal to $19$.  Similarly, for the two pencils in Table \ref{table:chainfamilies}, the generic N\'eron--Severi rank $\rho$ is at least $22-7=15$: but the divisor defined by $x_1=0,x_2^2=ix_3^2$ for either choice of $i^2=-1$ is $\SL(F_A)$-invariant, so the generic N\'eron--Severi rank $\rho$ is in fact at least $16$.  Now a specialization result due to Charles \cite{CharlesPic} shows that the rank over $\overline{\F_q}$ is always at least $\rho$ and is infinitely often equal to $\rho$ if the rank is even and infinitely often $\rho+1$ if the rank is odd.  By the first paragraph of this section, we conclude that the generic N\'eron--Severi rank of these two pencils is exactly $16$.  

The complete N\'eron--Severi lattice of rank $19$ for the case of the Dwork pencil $\Fsf_4$ is worked out via transcendental techniques by Bini--Garbagnati \cite[\S 4]{bg}.  It would be interesting to compute the full N\'eron--Severi lattices for the remaining four plus two families; Kloosterman \cite{Kloosterman:preprint} has made some recent progress on this question and in particular has also shown (by a count of divisors) that the generic N\'eron--Severi rank is $16$ for the $\Csf_2\Fsf_2$ and $\Csf_2 \Lsf_2$ pencils.

We conclude by a discussion of some applications of Theorem~\ref{thm:K3thm} in the context of mirror symmetry.  Let $Y_\psi$ be the pencil of K3 surfaces mirror to quartics in $\mathbb{P}^3$ obtained by taking the quotient of $\Fsf_4$ by $(\mathbb{Z}/4\mathbb{Z})^2$ and resolving singularities. It can be viewed as the minimal resolution of the complete intersection  \cite{NS01, dAMS03}
$$
Z(xyz(x+y+z+4\psi w) - w^4) \subseteq \PP^4.
$$
A computation described by Kadir \cite[Chapter 6]{kadir} shows that for odd primes and $\psi \in \F_q$ with $\psi^4 \neq 1$,
\begin{equation} 
Z(Y_\psi,T)=\frac{1}{(1-T)(1-qT)^{19}(1-q^2T)R_{\psi,q}(T)}. 
\end{equation}

This calculation combined with Theorem~\ref{thm:K3thm} yields the following corollary.

\begin{cor}
There exists $r_0 \geq 1$ such that for all $q=p^r$ with $r_0 \mid r$ and $p \neq 2,5,7$ and all $\psi \in \F_q$ with $\psi^4 \neq 1$, we have 
\[ Z(X_{\diamond,\psi}/\F_{q^r},T)=Z(Y_{\psi}/\F_{q^r},T). \]
\end{cor}

In other words, for all $\psi \in \F_q$ with $\psi^4 \neq 1$, not only do we have the \defi{strong mirror relationship} 
\[ \#X_{\diamond,\psi}(\F_{q^r}) \equiv \#Y_{\psi}(\F_{q^r}) \pmod{q^r} \]
for all $\diamond \in \calF$ and $r \geq 1$ (see Wan \cite{wan}), but in fact we have equality 
\[ \#X_{\diamond,\psi}(\F_{q^r})=\#Y_{\psi}(\F_{q^r}) \]
for all $r$ divisible by $r_0$.  Accordingly, we say that the zeta functions $Z(X_{\diamond,\psi}/\F_q,T)$  for all $\diamond \in \calF$ and $Z(Y_{\psi}/\F_q,T)$ are potentially equal, that is, equal after a finite extension.

In addition, quite concretely, Elkies--Sch\"utt \cite{ES} find an elliptic fibration on the mirror $Y_{\psi}$ that allow us to obtain more information about the factor $R_{\psi,q}(T)$. Via a Shioda--Inose structure, $Y_{\psi}$ corresponds to the abelian surface $E \times E'$ where $E,E'$ are elliptic curves with $j$-invariants $j,j'$ where
\[ jj' = (\mu+144)^3, \quad (j-1728)(j'-1728)=\mu(\mu-648)^2, \]
and $\mu=256\psi^4$.  The curves $E,E'$ are $2$-isogenous, and so are parametrized by the modular curve $X_0(2)/\langle w_2 \rangle$.  It follows that letting
\[ a_{\psi,q}=q+1-\#E(\F_q), \quad a'_{\psi,q}=q+1-\#E'(\F_q) \] 
then $a_{\psi,q}=\pm a'_{\psi,q}$. By factoring
\[ 1 - a_{\psi,q}T+qT^2=(1-\alpha_{\psi,q}T)(1-\beta_{\psi,q} T) \]
we have
\begin{equation} 
R_{\psi,q}(T)=(1-qT)(1-(a_{\psi,q}^2-2q)T + q^2T^2)=(1-qT)(1-\alpha_{\psi,q}^2T)(1-\beta_{\psi,q}^2T).
\end{equation}


\begin{thebibliography}{DKSSVW16}

\bibitem[AS89]{AS}A.~Adolphson and S.~Sperber, \emph{Exponential sums and {N}ewton polyhedra: cohomology and estimates}, {Ann.~of Math.}~(2), \textbf{130} (1989), 367--406.

\bibitem[AS08]{AS08} A.~Adolphson and S.~Sperber, \emph{On the zeta function of a projective complete intersection}, Illinois J.~Math.~\textbf{52} (2008), no.~2, 389--417.

\bibitem[AS16]{AS16}A.~Adolphson and S.~Sperber, \emph{Distinguished-root formulas for generalized Calabi--Yau hypersurfaces}, \texttt{arXiv:1602.03578}, 2016.

\bibitem[AP15]{AP15} M. Aldi, A. Peruni\v{c}i\'c. \emph{$p$-adic Berglund-H\"ubsch duality.} Adv. Theor. Math. Phys. 19 (2015), no. 5, 1115-1139. 

\bibitem[ABS14]{ABS14} M. Artebani, S. Boissi\`ere, A. Sarti. \emph{The Berglund-H\"ubsch-Chiodo-Ruan mirror symmetry for K3 surfaces,} Jour. Math. Pure. Appl. 102 (2014), pp. 758-781.

\bibitem[BH93]{BH93}  P. Berglund and T. H\"ubsch, \emph{A Generalized Construction of Mirror Manifolds}, Nuclear Physics B, vol 393, 1993.

\bibitem[Beu08]{Beu}
F. Beukers, \emph{Hypergeometric functions in one variable}, Notes, 2008, available at \verb|https://www.staff.science.uu.nl/~beuke106/springschool99.pdf|.

\bibitem[BCM15]{BCM}
F. Beukers, H. Cohen, and A. Mellit, \emph{Finite hypergeometric functions}, \verb|arXiv:1505.02900v1|, 2015.

\bibitem[BG14]{bg}
G. Bini and A. Garbagnati, \emph{Quotients of the {D}work pencil}, J. Geom. Phys. \textbf{75} (2014), 173--198.

\bibitem[BvGK12]{BvGK}
G. Bini, B. van Geemen, T. L. Kelly. \emph{Mirror quintics, discrete symmetries and Shioda maps}, J. Alg. Geom. \textbf{21} (2012), 401-412.

\bibitem[BCP97]{Magma} W.~Bosma, J.~Cannon, and C.~Playoust, \emph{The Magma algebra system.\ I.\ The user language}, J.\ Symbolic Comput.\ \textbf{24} (3--4), 1997, 235--265.

\bibitem[CD08]{candelas} P. Candelas, X. de la Ossa, \emph{The Zeta-function of a p-adic manifold, {D}work theory for Physicists}, arxiv:0705.2056v1, 2008. 

\bibitem[CDGP91]{CDGP}
P. Candelas, X. C.~de la Ossa, P. S.~Green, L. Parkes, \emph{A pair of Calabi--Yau manifolds as an exactly soluble superconformal theory}, Nuclear Physics B \textbf{359} (1991), no.~1, 21--74.

\bibitem[CDRV00]{CORV}
P. Candelas, X. de la Ossa,  F. Rodriguez Villegas, \emph{Calabi--Yau manifolds over finite fields, {I}}, arXiv:hep-th/0012233v1, 2000.  

\bibitem[CDRV01]{CORV2} P. Candelas, X. de la Ossa, F. Rodriguez-Villegas, \emph{Calabi--Yau manifolds
over finite fields {II}}, in \emph{Calabi--Yau varieties and mirror symmetry},
Toronto 2001, 121-157, hep-th/0402133.

\bibitem[Cha13]{Charles}
F. Charles, \emph{The Tate conjecture for \textup{K3} surfaces over finite fields}, Invent.~Math.~\textbf{194} (2013), no.~1, 119--145.

\bibitem[Cha14]{CharlesPic}
Fran\c{c}ois Charles, \emph{On the Picard number of \textup{K3} surfaces over number fields}, Algebra Number Theory~\textbf{8} (2014), no.~1, 1--17.

\bibitem[CR11]{CR11} A. Chiodo, Y. Ruan. \emph{LG/CY correspondence: the state space isomorphism} Adv. Math., 227, Issue 6 (2011), 2157-2188.

\bibitem[dAMS03]{dAMS03} P. L. del Angel, S. M\"uller-Stach. \emph{Picard--Fuchs equations, integrable systems and higher algebraic K-theory.}  Calabi--Yau varieties and mirror symmetry (Toronto, ON, 2001), 43-55,  Fields Inst. Commun., 38, Amer. Math. Soc., Providence, RI, 2003. 

\bibitem[Dol82]{Dol82} I. Dolgachev. \emph{Weighted projective varieties.} Group actions and vector fields (Vancouver, B.C., 1981), 34-71, Lecture Notes in Math., 956, Springer, Berlin, 1982.

\bibitem[DG11]{DG}C.~F. Doran and R.~S. Garavuso, \emph{Hori-{V}afa mirror periods, {P}icard-{F}uchs equations, and {B}erglund-{H}\"ubsch-{K}rawitz duality}, Journal of High Energy Physics (2011), issue 10, 128, 21 pp.

\bibitem[DGJ08]{DGJ08} C.~F. Doran, B. Greene, and S. Judes, \emph{Families of Quintic Calabi--Yau 3-folds with Discrete Symmetries} Comm. Math. Phys. \textbf{280} (2008) pp. 675-725.

\bibitem[DKSSVW17]{paperB}C.~F. Doran, T.~L. Kelly, A. Salerno, S. Sperber, J. Voight, and U. Whitcher, \emph{Hypergeometric properties of symmetric K3 quartic pencils}.  Preprint, 2017.

\bibitem[Dwo62]{Dwo62} B. Dwork, \emph{A deformation theory for the zeta function of a hypersurface}. 1963 Proc. Internat. Congr. Mathematicians (Stockholm, 1962), 247--259.

\bibitem[Dwo69]{padic} B. Dwork, \emph{{$p$}-adic cycles}, Inst. Hautes \'Etudes Sci. Publ. Math. \textbf{37} (1969) 27--115.

\bibitem[Dwo89]{uniqueness}
B. Dwork, \emph{On the uniqueness of Frobenius operator on differential equations}, Algebraic number theory,
Adv. Stud. Pure Math., vol.\ 17, Academic Press, Boston, MA, 1989, 89--96.

\bibitem[ES08]{ES}
N. D.~Elkies, M. Sch\"utt, \emph{\textup{K3} families of high Picard rank}, unpublished notes, 2008.

\bibitem[EG-Z16]{EGZ}
W. Ebeling, S. M.\ Gusein-Zade, \emph{Orbifold zeta functions for dual invertible polynomials}, Proc.\ Edinb.\ Math.\ Soc.\ (2) \textbf{60} (2016), no.\ 1, 99--106. 

\bibitem[FJR13]{FJR13} H. Fan, T. Jarvis, Y. Ruan. \emph{The Witten equation, mirror symmetry and quantum singularity theory}, Ann. of Math. (2) {\bf 178} (2013) no. 1, 1-106.

 \bibitem[FW06]{fw}L. Fu and D. Wan, \emph{Mirror congruence for rational points on {C}alabi-{Y}au varieties}, {Asian J. Math.} 10, 2006, 1, 1--10.

\bibitem[{G\"a}h11]{GahrsThesis}S. G{\"a}hrs, \emph{Picard-{F}uchs equations of special one-parameter families of invertible polynomials}, Ph.D. thesis, Gottfried Wilhelm Leibniz Univ. Hannover, arXiv:1109.3462.

\bibitem[{G\"a}h13]{Gahrs}S. G{\"a}hrs, \emph{Picard-{F}uchs equations of special one-parameter families of invertible polynomials} in \emph{Arithmetic and geometry of {K}3 surfaces and {C}alabi-{Y}au threefolds}, Fields Institute Communications \textbf{67}, Springer, New York, 2013, 285--310.

\bibitem[GP90]{GP}B.R. Greene and M. Plesser, \emph{Duality in {C}alabi-{Y}au moduli space}, Nuclear Physics {B} \textbf{338} (1990), no. 1, 15--37.

\bibitem[HN75]{HarderNara}
G. Harder and M. S. Narasimhan, \emph{On the cohomology groups of moduli spaces of vector bundles on curves}, Math.\ Ann.\ \textbf{212} (1975), 215--248.

\bibitem[Huy16]{Huybrechts}
D.\ Huybrechts, \emph{Lectures on K3 surfaces}, Cambridge Studies in Advanced Mathematics, vol.\ 158, Cambridge, 2016.

\bibitem[Kad04]{kadir} S.\ Kadir, \emph{The arithmetic of Calabi--Yau manifolds and mirror symmetry}, D.Phil. thesis, Univ. of Oxford, 2004. arXiv: hep-th/0409202


\bibitem[Kad06]{kadir2}S.\ Kadir, \emph{Arithmetic mirror symmetry for a two-parameter family of
              {C}alabi-{Y}au manifolds}, in \emph{Mirror symmetry. {V}}, AMS/IP Stud. Adv. Math., \textbf{38}, 35--86, Amer. Math. Soc., Providence, RI, 2006.

\bibitem[Kat68]{Kat68} N.\ Katz, \emph{On the differential equations satisfied by period matrices}, Inst.\ Hautes \'Etudes Sci.\ Publ.\ Math. \textbf{35} (1968), 223--258. 

\bibitem[Kat72]{Kat72} 
N.\ Katz, \emph{Algebraic solutions of differential equations ($p$-curvature and the Hodge filtration)}. Invent. Math. \textbf{18} (1972), 1--118. 

\bibitem[Kat90]{Katz:ESDE}
N. M.\ Katz, \emph{Exponential sums and differential equations}, Princeton University Press, Princeton, 1990.

\bibitem[Kat09]{katz:dwork}
N.~M. Katz,
\newblock \emph{Another look at the {D}work family},
\newblock { Algebra, arithmetic, and geometry: in honor of Yu. I. Manin}, Progr. Math., 270, 
\newblock Birkh\"{a}user Boston, Inc., 2009, 89--126.
   
\bibitem[KP15]{KimPera} W. Kim, K. Madapusi Pera, \emph{$2$-adic integral canonical models and the Tate conjecture in characteristic $2$}, \verb|arXiv:1512.02540v1|, 8 December 2015.

\bibitem[Kl17]{Kloosterman:preprint}
R. Kloosterman, \emph{Monomial deformations of Delsarte hypersurface}, preprint, 2017.

\bibitem[Kra09]{Kra09} M. Krawitz. \emph{FJRW rings and Landau--Ginzburg Mirror Symmetry}. arxiv: 0906.0796.   

\bibitem[KS92]{KS}
M.\ Kreuzer and H.\ Skarke, \emph{On the classification of quasihomogeneous functions}, Comm.\ Math.\ Phys.\ \textbf{150} (1992), no.\ 1, 137--147.

\bibitem[Lev99]{eightfold}S. Levy, ed., \emph{{T}he eightfold way}, Mathematical Sciences Research Institute Publications, Cambridge University Press, Cambridge, 1999.

\bibitem[MW16]{mw}C. Magyar and U. Whitcher.  \emph{Strong arithmetic mirror symmetry and toric isogenies}.  To appear in \emph{Proceedings of the AMS Special Session on Higher Genus Curves and Fibrations of Higher Genus Curves in Mathematical Physics and Arithmetic Geometry}.  arXiv:1610.01011

\bibitem[Maz72]{Mazur} B. Mazur, \emph{Frobenius and the {H}odge filtration}, Ann. of Math.~(2) \textbf{98} (1973) 58--95.

\bibitem[Miy15]{Miyatani}
K. Miyatani, \emph{Monomial deformations of certain hypersurfaces and two hypergeometric functions}.  Int. J. Number Theory \textbf{11} (2015), no.\ 8, 2405--2430.

\bibitem[Muk88]{mukai}
S. Mukai, \emph{Finite groups of automorphisms and the Mathieu group}. Inventiones Math. \textbf{94} (1988).

\bibitem[OZ02]{OZ}
K. Oguiso, and D.-Q. Zhang, \emph{The simple group of order 168 and \textup{K3} surfaces}, in \emph{Complex geometry ({G}\"ottingen, 2000)} (2002), 165-184.

\bibitem[NS01]{NS01} N. Narumiya; H. Shiga, \emph{The mirror map for a family of $K3$ surfaces induced from the simplest 3-dimensional reflexive polytope.} 
Proceedings on Moonshine and related topics (Montr\'eal, QC, 1999), 139-161, CRM Proc. Lecture Notes, 30 (2001), Amer. Math. Soc.


\bibitem[Per15]{Pera}
K. Madapusi Pera, \emph{The Tate conjecture for \textup{K3} surfaces in odd characteristic}, Invent.~Math.~\textbf{201} (2015), no.~2, 625--668.


\bibitem[Sab05]{Sabbah}
Claude Sabbah, \emph{Hypergeometric differential and $q$-difference equations}, 
\verb|http://www.cmls.polytechnique.fr/perso/sabbah/exposes/sabbah_lisbonne05.pdf|, 2005.

\bibitem[Shi86]{Shi86} T. Shioda, \emph{An explicit algorithm for computing the Picard number of certain algebraic surfaces}, Amer. J. Math. 108 (1986) 415-432.

\bibitem[SV13]{SV} 
Steven Sperber and John Voight, \emph{Computing zeta functions of nondegenerate hypersurfaces with few monomials}, LMS J.\ Comp.\ Math.\ \textbf{16} (2013), 9--44.

\bibitem[vanL07]{vL}
R. van Luijk, \emph{\textup{K3} surfaces with Picard number one and infinitely many rational points}, Algebra Number Theory \textbf{1} (2007), no.~1, 1--15.

\bibitem[Wan06]{wan}Daqing Wan, \emph{Mirror symmetry for zeta functions}, in \emph{Mirror symmetry.\ {V}}, AMS/IP Stud. Adv. Math., 38, 2006.

\bibitem[Yu08]{yu}Jeng-Daw Yu, \emph{Variation of the unit root along the Dwork family of Calabi--Yau varieties}, Math.\ Ann.\ \textbf{343} (2009), no.~1, 53--78.

\end{thebibliography}
\end{document}